\documentclass[11pt]{article}

\usepackage{amsmath,verbatim}
\usepackage{amsfonts}
\usepackage{amssymb}
\usepackage{amsthm,epic}
\usepackage{mathrsfs}
\usepackage{xcolor}
\usepackage{hyperref}

\usepackage{tikz} 
\usepackage[T1]{fontenc}
\usepackage{enumitem}
\usepackage{graphicx, xcolor}
\usepackage{footnote,pifont,dsfont}
\usetikzlibrary{patterns}

\renewcommand{\geq}{\geqslant}
\renewcommand{\leq}{\leqslant}
\renewcommand{\epsilon}{\varepsilon}
\newcommand{\tn}[1]{#1}

\newtheorem{thm}{Theorem}
\newtheorem{cor}[thm]{Corollary}
\newtheorem{ex}[thm]{Example}
\newtheorem{proposition}[thm]{Proposition}
\newtheorem{lem}[thm]{Lemma}

\theoremstyle{definition}
\newtheorem{defn}[thm]{Definition}
\newtheorem{conjecture}[thm]{Conjecture}

\theoremstyle{remark}

\newtheorem{rem}[thm]{Remark}
\usepackage{array}
\usepackage[font={small}]{caption}

\theoremstyle{remark}
\newtheorem*{example}{Example}

\usepackage[font={small,it}]{caption}

\colorlet{darkgreen}{green!50!black}
\definecolor{darkseagreen}{rgb}{0.56, 0.74, 0.56}
\definecolor{lightcyan}{rgb}{0.88, 1.0, 1.0}
\definecolor{lightblue}{rgb}{0.68, 0.85, 0.9}
\definecolor{palecerulean}{rgb}{0.61, 0.77, 0.89}

\newcommand{\mS}{\ensuremath{\mathcal{S}}}
\newcommand{\mR}{\ensuremath{\mathcal{R}}}

\newcommand{\mN}{\ensuremath{\mathcal{N}}}
\newcommand{\mV}{\ensuremath{\mathcal{V}}}

\newcommand{\ox}{\ensuremath{\overline{x}}}

\newcommand{\oy}{\ensuremath{\overline{y}}}

\newcommand{\mD}{\ensuremath{\mathcal{D}}}

\newcommand{\bzer}{\ensuremath{\mathbf{0}}}

\newcommand{\bt}{\mbox{\boldmath$\theta$}}

\newcommand{\oa}{\overline{a}}
\newcommand{\ob}{\overline{b}}

\newcommand{\mH}{\mathcal{H}}
\newcommand{\xx}{{x_\ell}}
\newcommand{\yy}{{y_\ell}}
\newcommand{\xs}{{x_s}}
\newcommand{\ys}{{y_s}}

\usepackage[T1]{fontenc}
\usepackage[mathscr]{eucal}
\usepackage{fullpage}

\colorlet{darkgreen}{green!50!black}

\date{May 2017}
\author{J. Courtiel\thanks{Department of Mathematics, University of British Columbia, Canada; \texttt{courtiel@lipn.univ-paris13.fr}}, S. Melczer\thanks{Cheriton School of Computer Science, University of Waterloo, Canada \& \'Ecole Normale Sup\'erieure de Lyon, France; \texttt{smelczer@uwaterloo.ca}}, M. Mishna\thanks{Department of Mathematics, Simon Fraser University, Canada; \texttt{mmishna@sfu.ca}} \& K. Raschel\thanks{CNRS \& Universit\'e de Tours, France; \texttt{kilian.raschel@univ-tours.fr}}}

\title{Weighted Lattice Walks and Universality Classes}
\begin{document}
\maketitle

\begin{abstract}
In this work we consider two different aspects of weighted walks in cones. To begin we examine a particular weighted model, known as the Gouyou-Beauchamps model. Using the theory of analytic combinatorics in several variables we obtain the asymptotic expansion of the total number of Gouyou-Beauchamps walks confined to the quarter plane. Our formulas are parametrized by weights and starting point, and we identify six different asymptotic regimes (called universality classes) which arise according to the values of the weights.  The weights allowed in this model satisfy natural algebraic identities permitting an expression of the weighted generating function in terms of the generating function of unweighted walks on the same steps. The second part of this article explains these identities combinatorially for walks in arbitrary cones and dimensions, and provides a characterization of universality classes for general weighted walks.  Furthermore, we describe an infinite set of models with non-D-finite generating function. 
\end{abstract}

\section{Introduction}
\subsection{Enumeration of walks in cones}
In recent years many enumeration formulas for lattice walks in cones have been discovered. These results have focused largely on unweighted quadrant walks with small steps, and a wide variety of techniques have been developed and adapted to treat them~\cite{BoMi10,FaRa12,BoRaSa14,MeWi16}. Ongoing works consider generalizations of this model, such as walks in higher dimensions~\cite{BoBoKaMe16}, walks with longer steps~\cite{FaRa15,BMBoMe16+}, and walks with weights~\cite{Fe14,KaYa15}. A lattice path model is comprised of a finite set of vectors, or steps, denoted $\mS$, and a cone in which the walks are contained, very often $\mathbb{R}_{\geq0}^2$. An important quantity here is the \emph{drift\/} of the model, which is the vector sum of the step set; in dimension two the drift is given by $\vec{d}=\sum_{(s_1,s_2) \in \mS} (s_1,s_2)$.

Early studies of walks on lattices~\cite{Krew65,Ge86} used ad hoc techniques, in the sense that the arguments presented were confined to a specific model or small collection of models. More recently, most analyses start from the multivariate generating function tracking the length of a walk and its end point. For some fixed cone and given model~$\mS$, we define~$e_{(i,j) \rightarrow (k,\ell)}(n)$ to be the number of walks in the cone starting at~$(i,j)$, ending at~$(k,\ell)$, and taking~$n$ steps, each of which comes from~$\mS$. One can then build the following series, each an element of $\mathbb{Q}[x,y][\![t]\!]$: 
\begin{align}
Q(x,y;t)&= \sum_{k,\ell,n\geq 0} e_{(0,0)\rightarrow (k,\ell)}(n)\, x^k\,y^\ell t^n, \label{eq:GF_arbitrary_starting_point0} \\
Q^{i,j}(x,y;t)&= \sum_{k,\ell,n\geq 0} e_{(i,j)\rightarrow (k,\ell)}(n) \, x^ky^\ell t^n.\label{eq:GF_arbitrary_starting_point}
\end{align}

The combinatorial decomposition of a walk into a shorter walk followed by a permissible step translates into a useful generating function equation. Several different approaches have been successfully applied to analyze this equation, including techniques from computer algebra, complex analytic methods, and skillful direct algebraic manipulation of the generating function equations.

Our main focus is on the asymptotic enumeration of these walks. Using probabilistic techniques, Denisov and Wachtel~\cite{DeWa15} recently determined formulas for walks in general cones that can be adapted to determine asymptotic behaviour in the case of models with zero drift (that is, when each coordinate of the drift is zero).  When the drift is non-zero, asymptotic estimates of the total number of walks are not known at the same level of generality, although there has been much recent activity for models in the quarter plane.  Singularity analysis of generating function expressions has led to new results~\cite{FaRa12, MeMi14a}, in addition to advances on the work of Denisov and Wachtel~\cite{DeWa15, Du14, GaRa16}. The techniques of Pemantle and Wilson~\cite{PemantleWilson2013} on analytic combinatorics in several variables, combined with expressions for walk generating functions as rational diagonals, have yielded explicit asymptotics for a large collection of models~\cite{MeMi16, MeWi16}, and we take a similar approach to weighted models in this work.

\subsection{Weighted walks in cones}
Given any (unweighted) model defined by a set of steps $\mS$, one can add weights associated to each step and ask what happens to asymptotic growth, the analytic or algebraic nature of the generating function, and other related questions. Formally, if we assign a weight $a_{s_1,s_2}$ to each step $(s_1,s_2)$ belonging to a model $\mathcal S$, then the weighted analogues of Equations~\eqref{eq:GF_arbitrary_starting_point0} and \eqref{eq:GF_arbitrary_starting_point} are
\begin{align}
Q_{\mathfrak a}(x,y;t)&= \sum_{\substack{w \textrm{ walk starting at }(0,0)\\\textrm{ending at }(k,\ell) \textrm{ of length }n \\ \textrm{staying in the cone}}} \, \prod_{\substack{ (s_1,s_2)\textrm{ step in }w \\ \textrm{(with multiplicity)}}} \hspace{-22pt} a_{s_1,s_2} \hspace{15pt} \, x^ky^\ell t^n, \label{eq:GFweighted} \\
Q^{i,j}_{\mathfrak a}(x,y;t)&= \sum_{\substack{w \textrm{ walk starting at }(i,j)\\\textrm{ending at }(k,\ell) \textrm{ of length }n \\ \textrm{staying in the cone}}}\, \prod_{\substack{ (s_1,s_2)\textrm{ step in }w \\ \textrm{(with multiplicity)}}} \hspace{-22pt} a_{s_1,s_2} \hspace{15pt} \,  x^ky^\ell t^n.\label{eq:GFweighted2}
\end{align}
 
Generally speaking, assigning weights to steps enables a better understanding of the different types of model behaviour. In fact, by continuously varying these weights one can often illustrate sharp transitions (or \emph{phase changes}) in asymptotic behaviour.

Initially, this framework---which is close to a probabilistic viewpoint~\cite{FaIaMa99} in which the weights are interpreted as transition probabilities---does not force a restriction on any particular choice of weights.  In order to obtain a detailed analysis, however, some restriction on the model weighting is currently necessary.  Here we focus on \emph{central weights}: by definition, a weighting is central if all paths with the same length, start and end points have the same weight.  Uniformly weighting each step gives a central weighting, but non-trivial examples also exist.  For instance, if we respectively assign the weights $1/a$, $a$, $b/a$, $a/b$ to the step set $\{ (-1,0), (1,0), (-1, 1), (1,-1)\}$, as prescribed in Figure~\ref{fig:GB_weights}, then we have defined a central weighting. In Section~\ref{sec:central} we show that every central weighting on this set of steps has this form, up to a uniform scaling of weights.

\begin{figure}\center
\includegraphics[width=.12\textwidth]{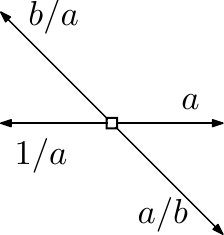}
\caption{Generic central weighting of the Gouyou-Beauchamps step set. Note that $(b/a)(a/b)=1=a/a$.}
\label{fig:GB_weights}
\end{figure} 

Central weightings have nice properties, including allowing one to express the generating function of a generic weighted model in terms of the generating function of the underlying unweighted model with weighted variables. Central weights are also related to the probabilistic notion of Cram\'er transform (which corresponds to an exponential change of measure).

\subsection{Article structure}
This article is organized into two parts. The first part is dedicated to the determination of exact asymptotic formulas for Gouyou-Beauchamps models (GB models for short) with central weightings. Our main result is Theorem~\ref{thm:main_asymptotic_result}, which we announce in the next section before considering some important consequences; Section~\ref{sec:ACSV} contains the proof of Theorem~\ref{thm:main_asymptotic_result}. Dominant asymptotic behaviour of a weighted GB model falls into one of several templates, depending on some simple relations between $a$ and $b$. We group the weighted models into \emph{universality classes\/} of models that share the same critical exponent. Figure~\ref{fig:drift_parameters_cases} illustrates the universality classes as a function of the weight parameters.

\begin{figure}\center
\includegraphics[width=0.5\textwidth]{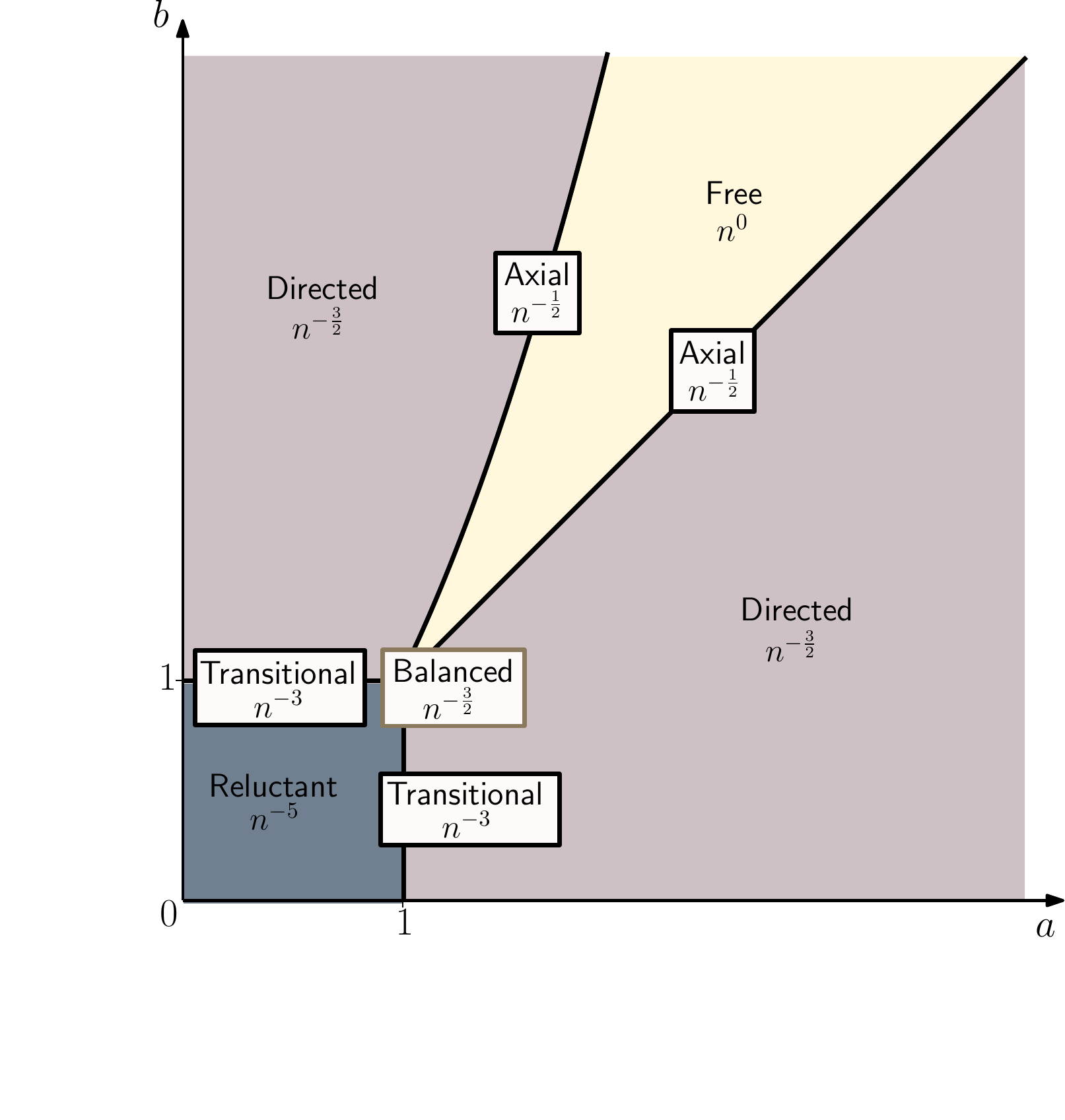}
   \vspace{-5mm}
	\caption{The polynomial growth ($n^{-\alpha}$ in Theorem~\ref{thm:main_asymptotic_result}) for the weighted GB model given as a function of the  weight parameters $a$ and $b$. Each region corresponds to a universality class. The boundaries of the $2$-dimensional regions, and their intersection form distinct regions.}
	\label{fig:drift_parameters_cases}
\end{figure}

Our results concerning general central weightings may be found in Part~\ref{part:CentralWeights}, and can be summarized as follows.
\begin{enumerate}
	\item For any model we determine which weightings of its steps are central; the number of parameters defining a central weighting is always equal to the dimension plus one. (One of these parameters corresponds to a uniform scaling of weights, meaning it can be specialized to $1$ since all of the properties we are interested in are preserved by scaling and one can recover asymptotics after this specialization.)  One should compare these results to the work of Kauers and Yatchak~\cite{KaYa15} which computationally determined families of weights leading to a D-finite generating function.  Except for one case, the weighted D-finite models they found correspond to central weights associated to specific D-finite models.
  	\item For these particular weights, we express (exactly and asymptotically) the numbers of weighted excursions in terms of the unweighted ones. These relations have strong consequences at the generating function level, and connect the (relatively unstudied) weighted and (well studied) unweighted generating functions. This is developed in Section~\ref{sec:central}.
  	\item For different weightings of the same underlying step set, one can determine information about the sub-families that have a common sub-exponential growth in the asymptotic formulas. The classes that we consider here are completely defined by drift. Studying these universality classes permits a direct connection between the ACSV enumerative approach and the probabilistic approach. This is developed in Section~\ref{sec:Diagrams}.
  	\item Finally, thanks to the above systematic construction of central weights, we are able to identify an infinite number of new models with non-D-finite generating functions. 
\end{enumerate}

\part{The asymptotic enumeration of weighted Gouyou-Beauchamps walks}
\label{part:GB}

\section{The main result and its consequences}
\subsection{Theorem statement}
We begin with a close study of the centrally weighted models whose underlying step set is the Gouyou-Beauchamps model of Figure~\ref{fig:GB_weights}.  Applying our construction (to be given in Part~\ref{part:CentralWeights}) of central weights to this particular model we obtain the weights in Figure~\ref{fig:GB_weights}, which can be expressed with two parameters $a$ and $b$:
\begin{equation}
\label{eq:GB_weighting}
  (-1,0)\longmapsto 1/a, \qquad (1,0)\longmapsto a, \qquad (-1,1)\longmapsto b/a, \qquad (1,-1)\longmapsto a/b.
\end{equation}
The generic weighted model $\mS$ has corresponding drift 
\begin{equation} 
\vec{d}=(d_x,d_y) = \left(\frac{(1+b)(a^2-b)}{ab},\frac{(a+b)(b-a)}{ab}\right). \label{eq:drift} 
\end{equation} 
The drift is an important predictor of a model's asymptotics. For non-negative integers $i$ and $j$, recall that $Q_{\mathfrak a}^{i,j}(x,y;t)$ is the generating function~\eqref{eq:GFweighted2} of a GB walk starting from~$(i,j)$. 
\begin{thm}
\label{thm:main_asymptotic_result}
Fix constants $a,b>0$.  As $n\to\infty$ the number of weighted GB walks of length $n$, starting from $(i,j)$ and ending anywhere while staying in $\mathbb R^2_+$, satisfies
\begin{equation} \label{eq:main_asymptotic_result}
[t^n]Q_{\mathfrak a}^{i,j}(1,1;t) = \kappa\cdot V^{[n]}(i,j)\cdot \rho^n\cdot  n^{-\alpha}\cdot\left(1+O\left(\frac{1}{n}\right)\right),
\end{equation}
where the exponential growth $\rho$ and the critical exponent $\alpha$ are given by Table~\ref{tab:cases}. The universal constant $\kappa$ and the harmonic function $V^{[n]}(i,j)$ are presented in Appendix~\ref{sec:values_harmonic_function}; the harmonic function $V^{[n]}(i,j)$ can depend on the parity $[n]$ of $n$ but is otherwise constant with respect to $n$.
\end{thm}

\begin{table}\center
\begin{tabular}{llcl}
Class & Condition & $\rho$ & $\alpha$ \\ \hline 
{\bf Balanced} & $a=b=1$ & 4 & 2\\[+1mm]
{\bf Free} & $\sqrt{b}<a<b$ &  $\frac{(1+b)(a^2+b)}{ab}$ & 0 \\[+1mm]
{\bf Reluctant }&$a<1$ and $b<1$ & $4$ & 5\\[+1mm]
{\bf Directed} & $b>1$ and $\sqrt{b}>a$ &$\frac{2(b+1)}{\sqrt{b}}$&3/2\\[+1mm]
               & $a>1$ and  $a>b$ & $\frac{(1+a)^2}{a}$ & 3/2\\[+1mm]
{\bf Axial} &$b=a^2>1$ &  $\frac{2(b+1)}{\sqrt{b}}$ & 1/2 \\[+1mm] 
		& $a=b>1$   & $\frac{(1+a)^2}{a}$ & 1/2\\[+1mm]
{\bf Transitional} & $a=1,b<1$ or $b=1,a<1$ & $4$ & 3
\end{tabular}
\caption{The six different universality classes for weighted GB walks, under the weighting~\eqref{eq:GB_weighting}.}
\label{tab:cases}
\end{table}
We give names to the six regions in which the exponent $\alpha$ is constant, inducing six different \textit{universality classes} for the central GB walks. The \textit{free} region corresponds to walks where the drift is in the interior of the first quadrant, translating the idea that the walks behave as if they were unrestricted. The \emph{reluctant\/} case corresponds to the domain where $a < 1$ and $b < 1$, when the walks have the smallest exponential growth. The \emph{transitional\/} domain is the boundary of the reluctant region, while the \textit{axial} domain is the boundary of the free region. The drift being $(0,0)$ corresponds to the \textit{balanced} model, and what remains are the \textit{directed} models. All these regions are defined more generally for other models in Section~\ref{sec:Diagrams}.

\subsection{Main features and applications of Theorem~\ref{thm:main_asymptotic_result}}
We now give a list of remarks around Theorem~\ref{thm:main_asymptotic_result}, illustrating applications in enumerative combinatorics, analytic combinatorics, probability theory and potential theory.

\subsubsection{Combinatorial enumeration and uniform generation}
Most plainly, Theorem \ref{thm:main_asymptotic_result} is a combinatorial result. It gives the asymptotics of the total number of walks confined to the quarter plane for the weighted Gouyou-Beauchamps model. This model is so-named because Gouyou-Beauchamps~\cite{Gouy86} discovered a simple hypergeometric formula for the (unweighted) walks that return to an axis. As shown by Gouyou-Beauchamps~\cite{Gouy89}, this lattice path model encodes several combinatorial classes: the set of walks ending anywhere are in bijection with pairs of non-intersecting prefixes of Dyck paths, and the walks ending on the axis are in bijection with Young tableaux of height at most~$4$. This quarter plane lattice path model is also in bijection with the lattice path model having step set $\mS = \{(\pm1,0),(0,\pm1)\}$ which is restricted to the region $\{(x,y) : 0 \leq x \leq y \}$.  Thus, the GB model can be considered as a walk in a \emph{Weyl chamber} (see Gessel and Zeilberger~\cite{GeZe92} for information on Weyl chamber walks).

Asymptotic results are also key in random generation of structures. Lumbroso, Mishna and Ponty~\cite{LuMiPo16} describe an algorithm for randomly sampling lattice path models, with the expected running time of the algorithm dependent on the polynomial exponent $\alpha$; thus, models in the same universality class have the same cost of random generation. Uniformly sampled walks give a picture of how the weights impact the shape of the walk, and two examples from this algorithm are given in Figure~\ref{fig:GB_example}.  In the same vein, Bacher and Sportiello~\cite{Bacher2016} have described an anticipated rejection random sampling algorithm with linear complexity when the expected first exit time is equivalent to $C n^{-\alpha}$ with $\alpha \in [0,1)$. We can see that this occurs in our context for the axial and free cases.

\begin{figure}\center
\hfill\includegraphics[width=.3\textwidth]{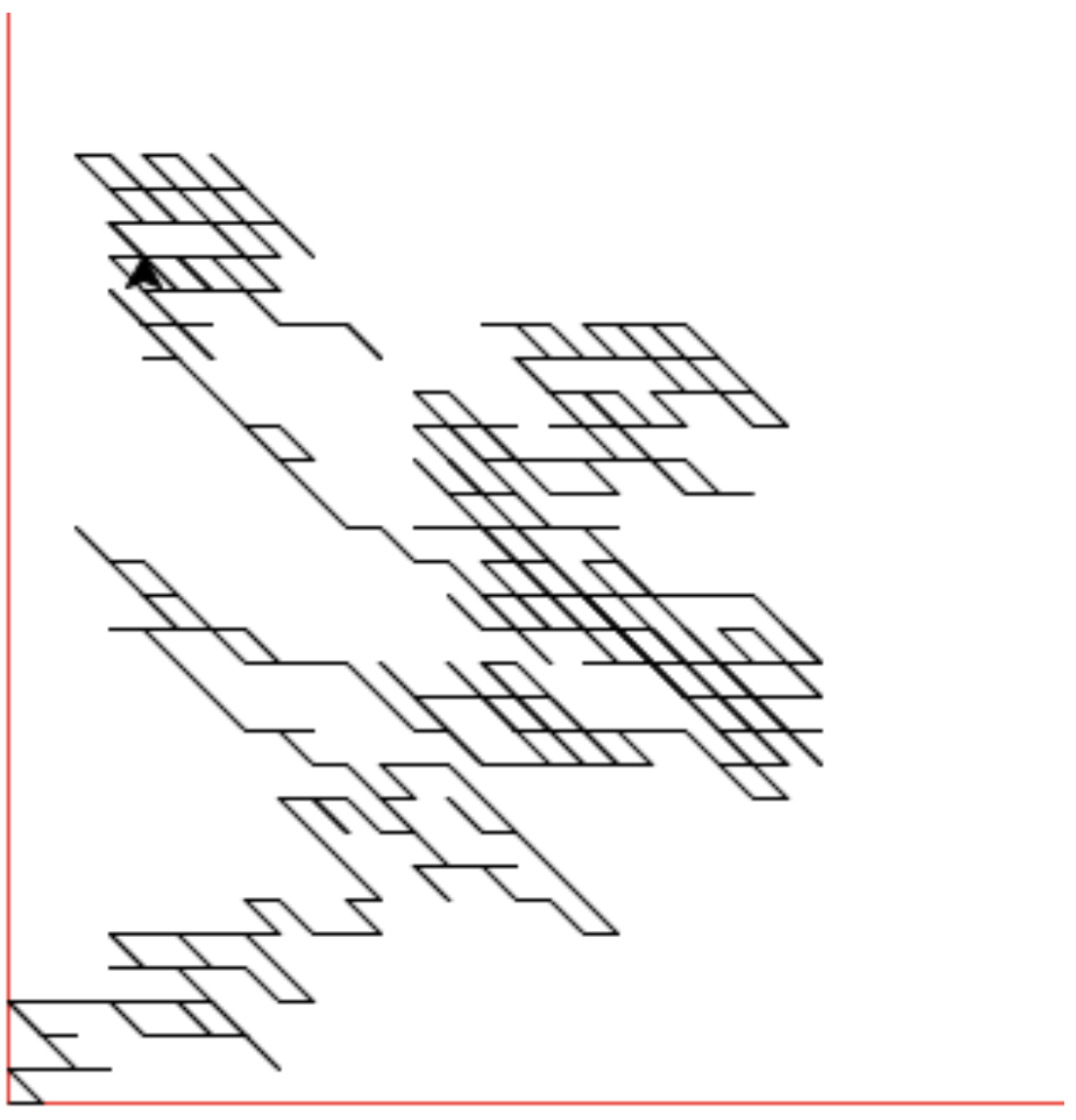}\hfill
\includegraphics[width=.3\textwidth]{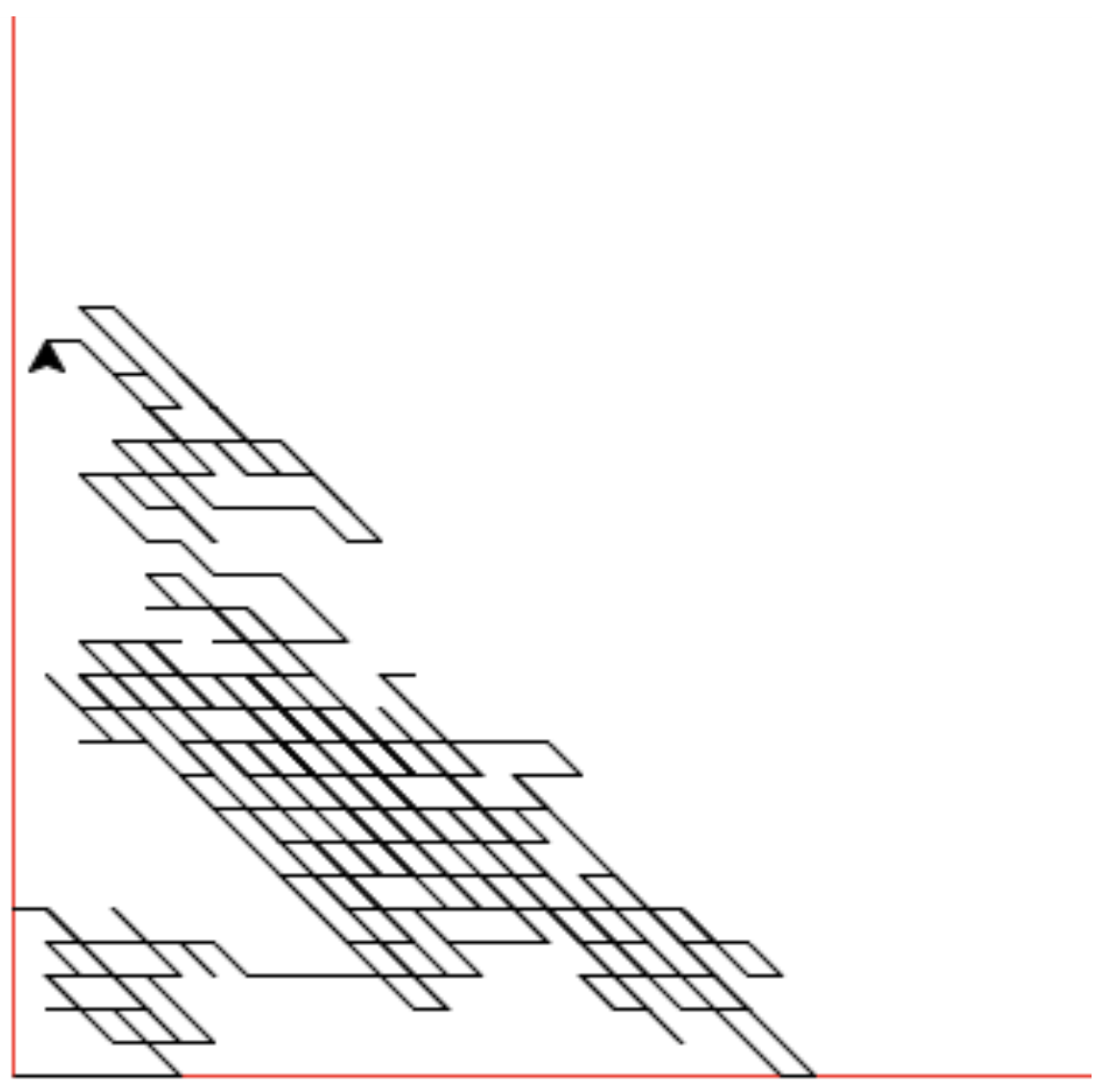}\hfill\mbox{}
\caption{Two uniformly sampled GB walks on 800 steps: an unweighted (balanced) model (left) and a weighted (reluctant) model with parameters $a=\frac{1}{\sqrt{2}}$ and $b=1$ (right).} 
\label{fig:GB_example}
\end{figure}

\subsubsection{Analytic Combinatorics of Several Variables}
Our results demonstrate the theory of analytic combinatorics in several variables (ACSV). Mainly developed by Pemantle and Wilson~\cite{PemantleWilson2002,PemantleWilson2004,PemantleWilson2008,PemantleWilson2010,PemantleWilson2013}, and collaborators, this theory is used to determine asymptotic formulas of the coefficients of a class of multivariate power series. Although several examples of applications to lattice path combinatorics have already been given in the literature---see, for example, Bressler et al.~\cite{BresslerGreenwoodPemantlePetkovsek2010} or Melczer and Wilson~\cite{MeWi16}---our study provides a continuous $2$-parameter family of examples, with many phase transitions (exhibited by the variety of exponents $\alpha$ that we obtain). We believe this offers a concrete and rich framework for ACSV, and sheds light on how the intermediary geometric objects drive the asymptotic behaviour.

\subsubsection{Excursions}
\label{sec:excursions}
Theorem~\ref{thm:main_asymptotic_result} also has consequences in probability theory. Indeed, for random walks with drift in cones, a general derivation of asymptotic expressions for the probability of a random walk to stay in a given cone up to a fixed time is still an open problem (see Garbit, Mustapha and Raschel~\cite{GaMuRa16} for work in this direction).

Generally speaking, (combinatorial) asymptotic results on excursions\footnote{An \emph{excursion\/} is a walk starting and ending at precise points. Very often it implies starting and ending at the origin.}  can be turned into (probabilistic) local limit theorems by the formula
\begin{equation}
\label{eq:link_excursions}
     {\mathbb P[(i,j)+S_n=(k,\ell),\tau>n]}=\frac{e_{(i,j) \rightarrow (k,\ell)}(n)}{\#\mathcal S^n},
\end{equation}
where $\tau$ denotes the first exit time of the random walk $(S_n)$ from the quadrant: 
\begin{equation}
\label{eq:def_exit_time}
     \tau=\inf\{n\geq0:S_n\notin\mathbb N^2\}. 
\end{equation}     
In the same way, results on walks with prescribed length can be written as
\begin{equation}
\label{eq:link_non-exit_probability}
     \mathbb P_{(i,j)}[\tau>n]=\frac{e_{(i,j) \rightarrow \mathbb N^2}(n)}{\#\mathcal S^n}=\frac{\sum_{(k,\ell)\in\mathbb N^2}e_{(i,j) \rightarrow (k,\ell)}(n)}{\#\mathcal S^n}.
\end{equation} 

Denisov and Wachtel~\cite{DeWa15} prove a local limit theorem for random walks in cones, giving the asymptotics~\eqref{eq:link_excursions} and, in the case of zero drift, obtain a precise estimate of the non-exit probability~\eqref{eq:link_non-exit_probability}. Duraj~\cite{Du14} established an asymptotic estimate of the non-exit probability for the walks in the \textit{reluctant} universality class (roughly speaking for the general case, the ones with negative drift). Furthermore, the exponential growth $\rho$ in Equation~\eqref{eq:main_asymptotic_result} is computed by Garbit and Raschel~\cite{GaRa16}: for a large class of cones and dimensions it is equal to the minimum of the Laplace transform of the increments of the random walk on the dual cone (note that the quarter plane is self dual). Further discussion on this topic occurs in Subsection~\ref{ss:gmr}.

Our contribution in this area is as follows. For a continuous family of processes---whose drift can be equal to any prescribed value---we find the exponents~$\alpha$ in Equation~\eqref{eq:main_asymptotic_result}. Interestingly, $\alpha$ can be equal to six different values: $0,1/2,3/2,2,3,$ or $5$.

We further observe a correspondence between the direction of the drift and the value of $\alpha$, as shown in Figure \ref{fig:drift}; this correlation between the drift and value of the exponent seems to become slightly more complex in the general case~\cite{GaMuRa16}. We discuss conditions which induce the correspondence in Section~\ref{sec:Diagrams}.

\subsubsection{Periodicity of leading terms} 
Interestingly, the leading term of the asymptotic expansion~\eqref{eq:main_asymptotic_result} can have periodic behaviour. More specifically, the term $V^{[n]}$ depends on the parity of $n$ in the directed,  transitional and reluctant cases. For general walks in cones, the periodicity of the leading term often exists for excursions---in periodic random walks, for instance---however when the drift is zero the leading term of the non-exit probability~\eqref{eq:link_non-exit_probability} is not periodic~\cite{DeWa15}. 

This phenomenon of periodicity, which only occurs in the case of a drift with at least one negative coordinate, is not fully understood. The reader can refer to Bostan~{et al.}~\cite[Tables 3--6]{BoChHoKaPe17} or Melczer and Wilson~\cite{MeWi16} for examples of periodic constants in the unweighted small step case.

\subsubsection{Discrete harmonic functions} 
An additional corollary of our results is a formula for a family of discrete harmonic functions. Indeed, the quantity $V^{[n]}(i,j)$ in Equation~\eqref{eq:main_asymptotic_result} satisfies a property of discrete $\rho$-harmonicity, since 
\begin{align}
\label{eq:harmonicity_V}
\begin{split}
\rho \cdot V^{[n+1]}(i,j)=(1/a)V^{[n]}(i-1,j)&+(b/a)V^{[n]}(i-1,j+1) \\
&+ aV^{[n]}(i+1,j)+(a/b)V^{[n]}(i+1,j-1).
\end{split}
\end{align}
Equation~\eqref{eq:harmonicity_V} follows from the recurrence relations satisfied by the numbers of walks $e_{(i,j) \rightarrow \mathbb N^2}(n)$, as well as the asymptotic estimate given in Theorem \ref{thm:main_asymptotic_result}.

Two different cases should be mentioned. If $V^{[n]}$ does not depend on $n$ (which is typically the situation), then $V=V^{[n]}$ is called a discrete harmonic function. Such discrete harmonic functions are of paramount interest: via the standard procedure of a Doob transform they are used to construct processes conditioned never to leave cones~\cite{DeWa15,De16,Du14,LeRa16}. These processes are important in probability theory, appearing in non-colliding processes and eigenvalues of certain random matrices, among other areas. See Section~\ref{subsubsec:conditioning_random_walks} for more details. In the second case, when $V^{[n]}$ depends on $n$, $V^{[n]}$ is called a \emph{caloric} function~\cite{BoMuSi15}.

Although the discrete recurrence~\eqref{eq:harmonicity_V} is reminiscent of the one counting walks~\eqref{eq:kernel}, the solution space of Equation~\eqref{eq:harmonicity_V} is rather different. In particular, there exist a priori infinitely many different harmonic functions solving the discrete recurrence~\cite{LeRa16} while Equation~\eqref{eq:harmonicity_V} uniquely defines a model's generating function.

Thanks to our asymptotic results we indirectly find discrete harmonic/caloric functions, listed in Appendix~\ref{sec:values_harmonic_function}, with particularly nice exponential and polynomial expressions. As an example, in the zero drift case $a=b=1$ one has 
\begin{equation}
\label{eq:universal_GB_harmonic function}
     V(i,j)=\frac{(i+1)(j+1)(i+j+2)(i+2j+3)}{6}.
\end{equation}
This is a polynomial expression of degree four, homogeneous in the variables $i+1,j+1$.  Additionally, discrete harmonic functions seem promising for the random generation of random walks confined to cones~\cite{Fu16}.

\subsubsection{Conditioning random walks}
\label{subsubsec:conditioning_random_walks}
A typical question in probability theory~\cite{DeWa15,De16,Du14,LeRa16} is to compare two different ways of conditioning a random walk $(S_n)$ to never leave a cone:
\begin{itemize}
    \item The first way is to condition on the event $\{\tau>n\}$, with $\tau$ being the first exit time~\eqref{eq:def_exit_time} from the cone, and to let $n$ go to infinity; the law of the conditioned Markov chain is characterized by
      \begin{equation}
      \label{eq:first_possibility_CRW}
      \widehat{\mathbb P}_{(i,j)}[S_1=(k,\ell)]=\lim_{n\to\infty}{\mathbb P}_{(i,j)}[S_1=(k,\ell)\vert\tau>n].
      \end{equation}
    This definition of conditioned random walks has a clear and natural probabilistic interpretation. However, except in the free case, the limit event $\{\tau=\infty\}$ has zero probability so that it is a priori unclear that the limit in the right-hand side of Equation~\eqref{eq:first_possibility_CRW} exists. 
    \item The second way consists in using a (discrete) harmonic function and the procedure of Doob transformation (see Section 13 in Doob~\cite[Chapter 6]{Do84}).  If $V$ is a $z$-harmonic function---i.e., $z\cdot V(i,j)={\mathbb E}_{(i,j)}[V(S_1),\tau>1]$---which is zero on the exterior of the cone\footnote{In the quadrant case, $V(i,j)=0$ whenever $i<0$ or $j<0$.}, then     
      \begin{equation}
      \label{eq:second_possibility_CRW}
      {\mathbb P}^V_{(i,j)}[S_1=(k,\ell)]={\mathbb P}_{(i,j)}[S_1=(k,\ell),\tau>1]\frac{V(k,\ell)}{z\cdot V(i,j)}
      \end{equation}
    defines a Markov chain which by construction stays in the cone. In Equation~\eqref{eq:second_possibility_CRW} the transition probabilities ${\mathbb P}_{(i,j)}[S_1=(k,\ell),\tau>1]$ are obtained from Equation~\eqref{eq:GB_weighting}. 
\end{itemize}
For random walks in cones with no drift, Denisov and Wachtel~\cite{DeWa15} showed that these two approaches coincide: $\widehat{\mathbb P}={\mathbb P}^V$. Similar considerations for negative drift (reluctant) random walks are given by Duraj~\cite{Du14}. 

To link back to our results, we compare these two approaches under the additional assumption that our harmonic function $V^{[n]}$ does not depend on the parity of $n$.  To show that the two viewpoints coincide for centrally weighted GB models we note that
\begin{align}
\label{eq:reformulation_conditional_probability}
\begin{split}
     \lim_{n\to\infty}{\mathbb P}_{(i,j)}[S_1=(k,\ell)\vert\tau>n]&=\lim_{n\to\infty}\mathbb P_{(i,j)}[S_1=(k,\ell),\tau>1]\frac{\mathbb P_{(k,\ell)}[\tau>n-1]}{\mathbb P_{(i,j)}[\tau>n]}\\
     &={\mathbb P}^V_{(i,j)}[S_1=(k,\ell)].
     \end{split}
\end{align}
The first equality in Equation~\eqref{eq:reformulation_conditional_probability} is simple and comes from the definition of a conditional probability together with the Markov property; the second equality follows from using Equations~\eqref{eq:main_asymptotic_result} and \eqref{eq:link_non-exit_probability} with $z=\rho/\#\mathcal S$. There are more subtleties in the case where $V^{[n]}$ depends on the parity of $n$, which we do not cover here.

A second interesting aspect of conditioned random walks is to prove that the following operations commute:
\begin{itemize}
     \item $n\to\infty$ in the conditioning by $\{\tau>n\}$,
     \item the drift $\vec{d}$ converges to a fixed drift $\vec{d}_0$.
\end{itemize}
Precisely, this amounts to proving that
\begin{equation}
\label{eq:two_limits_coincide}
     \lim_{\vec{d}\to\vec{d}_0}\lim_{n\to\infty}{\mathbb P}_{(i,j)}[S_1=(k,\ell)\vert\tau>n]=\lim_{n\to\infty}\lim_{\vec{d}\to\vec{d}_0}{\mathbb P}_{(i,j)}[S_1=(k,\ell)\vert\tau>n].
\end{equation} 
Equation~\eqref{eq:two_limits_coincide} can be interpreted as the continuity of the law of conditioned random walks with respect to the drift of the step set.  Despax~\cite{De16} considers, in any dimension but for a particular family of random walks, the case $\vec{d}_0=0$ and proves the above identity when $\vec{d}\to\vec{d}_0$ while remaining in the cone.

Using our results we can prove, for the GB model, equality in Equation~\eqref{eq:two_limits_coincide} for any value of the drift $\vec{d}_0 \in \mathbb{R}^2$ (not necessarily in the cone).  This is an easy consequence of the explicit expression of the harmonic function and Equation~\eqref{eq:reformulation_conditional_probability}.  Notice in particular that
\begin{equation*}
     \lim_{\vec{d}\to(0,0)}\frac{V(i,j)}{V(0,0)}=\lim_{(a,b)\to(1,1)}\frac{V(i,j)}{V(0,0)}=\frac{(i+1)(j+1)(i+j+2)(i+2j+3)}{6}
\end{equation*}
exists (i.e., does not depend on the way that $(a,b)\to(1,1)$). The right-hand side of the identity above, corresponding to Equation~\eqref{eq:universal_GB_harmonic function}, can thus be viewed as a universal harmonic function for the GB model.

\section{The generating function as a diagonal}
\label{sec:ACSV}

As mentioned above, we will derive our asymptotic results by representing the combinatorial sequences under consideration in terms of coefficient extractions of explicit rational functions.  In order to derive these representations we use an approach called the \emph{algebraic kernel method}\footnote{The algebraic kernel method, which relies on rational maps fixing a certain Laurent polynomial related to a lattice path model, is closely related to the more classical kernel method in which one manipulates roots of a Laurent polynomial (see Banderier and Flajolet~\cite{BaFl02} for a modern illustration of the classical kernel method).} which has been applied widely to lattice path problems in recent years~\cite{Bous02,BoPe03,BoMi10,BoBoKaMe16}; our presentation closely follows that of Bousquet-M\'elou and Mishna~\cite{BoMi10}, which is now standard.

A lattice path model is said to have \emph{small steps} if its set $\mS$ of allowable steps is a subset of $\{0,\pm1\}^2$.  In their work, Bousquet-M\'elou and Mishna proved that there are 79 non-trivial distinct small step quarter plane models (which are not equivalent to half plane models) and provided many expressions for their generating functions. A model is called \emph{singular} if there is a half-plane which contains all of its allowable steps; for instance the step set $\mS = \{(-1,1),(1,1),(1,-1)\}$ is non-trivial and singular. 

\subsection{The kernel equation}
As noted in the introduction, a straightforward functional equation for the generating function $Q(x,y;t)$ arises from the basic fact that a walk of length $n$ is either empty or a walk of length $n-1$ plus a valid step. In order to express this recurrence on the level of generating functions, we define the \emph{inventory} $S(x,y) = \sum_{(s_1,s_2)\in\mS}\,x^{s_1}\,y^{s_2}$ and the \emph{kernel\/} $K(x,y;t)= xy(1-tS(x,y))$. Note that for a non-trivial model with small steps the kernel is a polynomial.  Translating the recursive decomposition of a walk, and taking into consideration the boundary conditions, one obtains the functional equation
\begin{equation}
K(x,y;t)Q(x,y;t)= xy + K(x,0; t)Q(x,0;t)+ K(0,y; t)Q(0,y;t)- K(0,0; t)Q(0,0;t), \label{eq:kernel}
\end{equation}
which is proven in Bousquet-M\'elou and Mishna~\cite{BoMi10}.  This functional equation is known as the \emph{kernel equation} and is easily modified to handle walks with different starting position: the term $xy$ is replaced by $x^{i+1}y^{i+1}$ when the walk begins at $(i,j)$ instead of the origin. The subsequent analysis is robust enough to handle this modification, but as it becomes notationally heavy we mainly detail the case $i=j=0$.

\subsection{Series extractions and diagonals}
The kernel method associates to each step set a specific group $\mathcal{G}(\mS)$ of bi-rational transformations, generated by involutions, which fixes the kernel. Remarkably, for models with small steps the finiteness of this group is perfectly correlated to the nature of the multivariate generating functions $Q(x,y;t)$. In particular, for the 79 distinct non-trivial models with small steps $\mathcal{G}(\mS)$ is finite if and only if $Q(x,y;t)$ is D-finite. 
 
The group associated to the unweighted Gouyou-Beauchamps model has order eight and is generated by the two involutions
\[ \Psi:(x,y) \mapsto \left(\frac{y}{x}, y\right), \quad \Phi:(x,y) \mapsto \left(x, \frac{x^2}{y}\right).\] 
We shall see that this group does not substantially change under a central weighting, which is what allows a uniform diagonal expression for models under such weightings.

The kernel method applies the transformations of~$\mathcal{G}(\mS)$ to Equation~\eqref{eq:kernel} in order to generate new functional equations, which are then manipulated into an expression for $Q(x,y;t)$. Consider the ring $\mR := \mathbb{Q}((x,y))[[t]]$ of power series in $t$ whose coefficients are multivariate Laurent series in $x$ and $y$, with a generic element having the form
\begin{equation} F(x,y;t) = \sum_{n \geq 0} \left(\sum_{i,j \geq k_n} f_{i,j}(n)x^iy^j\right)t^n \label{eq:FinR} \end{equation}
for integers $k_0,k_1,k_2,\dots$. The expression for $Q(x,y;t)$ obtained by the kernel method uses the series extraction operator $[x^{\geq0}y^{\geq0}]$, which takes $F(x,y;t) \in \mR$ defined by Equation~\eqref{eq:FinR} and returns the multivariate power series 
\[ [x^\geq y^\geq]F(x,y;t)=\sum_{i,j,n\geq0} f_{i,j}(n) x^i y^j t^n. \]
Proposition 11 of Bousquet-M\'elou and Mishna~\cite{BoMi10} gives that the generating function for the GB model satisfies
\begin{equation}
Q(x,y;t)= [x^\geq y^\geq]\frac{\left(1- \ox\right)  \left(1+\ox\right) \left(1-\oy\right) \left(1-\ox^2y\right) \left(1-x\oy\right)\left(1+x\oy\right)}{1-t(x+\ox+x\oy+\ox y)}, \label{eq:Qpospart}
\end{equation}
where we write $\ox=1/x$ and $\oy=1/y$ to simplify notation.  

In order to determine the asymptotic behaviour of the number of lattice walks in a model we convert this expression for the multivariate generating function $Q(x,y;t)$ into one for the univariate generating function $Q(1,1;t)$.  The \emph{diagonal} operator $\Delta$ takes $F(x,y;t) \in \mR$ defined by Equation~\eqref{eq:FinR} and returns the univariate power series
\[ (\Delta F)(t)=\sum_{n \geq 0} f_{n,n}(n) t^n \]
whose coefficients are determined by the terms of $F(x,y;t)$ where the exponents of each variable are equal.  Equation~\eqref{eq:Qpospart} can be used to give a diagonal expression for $Q(1,1;t)$ as 
\[ [x^\geq y^\geq]F(x,y;t)\bigg|_{x=1,y=1} = \Delta\left( \frac{F\left(\ox,\oy,xyt\right)}{(1-x)(1-y)}\right) \]
whenever $F\left(\ox,\oy,xyt\right)$ is a well-defined element of $\mR$ (this fact follows from a straightforward manipulation of the power series in question; see Proposition 2.6 of Melczer and Mishna~\cite{MeMi16} for more information).  In the end, one obtains
\[ Q(1,1;t) = \Delta \left( \frac{(x+1)(\ox^2-\oy)(x-y)(x+y)}{1-xyt(x+x\oy+y\ox+\ox)} \right). \]
Following the same argument with weighted steps, where the weighted inventory 
\[ S(x,y) = \sum_{(s_1,s_2)\in\mS}(a^{s_1}b^{s_2})x^{s_1}y^{s_2} = ax + \frac{1}{ax} + \frac{ax}{by}+ \frac{by}{ax}\] 
is used, results in the diagonal expression
\begin{equation}
\label{eqn:weighted_extraction_nonan}
Q_{\mathfrak a}(1,1;t) = \Delta \left( \frac{(y-b)(a-x)(a+x)(a^2y-bx^2)(ay-bx)(ay+bx)}{a^4b^3x^2y(1-txyS(\ox,\oy))(1-x)(1-y)} \right),
\end{equation}
which is valid for any non-zero constants $a$ and $b$ (although we only consider the case when $a$ and $b$ are positive below). We can similarly deduce an expression for weighted walks parametrized by start point $(i, j)$ by modifying the kernel equation to account for these new starting points, as discussed above; only the numerator of the resulting rational diagonal depends on $i$ and $j$.  A weighted version of Equation~\eqref{eq:Qpospart} was also derived by Kauers and Yatchak~\cite{KaYa15} during their computational investigation into the analytic nature of weighted lattice path generating functions.

In order to determine asymptotics through analytic techniques, it is easiest to consider diagonals of rational functions with power series expansions.  As 
\[ \Delta\left((xyt)^{-k}F(x,y,t)\right) = t^{-k}\Delta F(x,y,t)\] 
for any natural number $k$ and $F \in \mR$, we can modify Equation~\eqref{eqn:weighted_extraction_nonan} to obtain
\begin{equation}
\label{eqn:weighted_extraction}
Q_{\mathfrak a}(1,1;t) = \frac{1}{a^4b^3t^2} \Delta \left( \frac{yt^2(y-b)(a-x)(a+x)(a^2y-bx^2)(ay-bx)(ay+bx)}{(1-txyS(\ox,\oy))(1-x)(1-y)} \right).
\end{equation}

We prove Theorem~1 by representing the coefficients of this diagonal as a multivariate complex integral, and asymptotically approximating the integral expression.  This is possible due to a multivariate generalization of Cauchy's Integral Formula (CIF).

\begin{thm}[Multivariate CIF]
Suppose that $F(x,y,t) \in \mathbb{Q}(x,y,t)$ is analytic at $(0,0,0)$ with a power series expansion $F(x,y,t) = \sum_{i,j,n \geq 0} f_{i,j}(n) x^{i}y^{j}t^{k}$ at the origin.  Then for all $n \geq 0$,
\begin{equation} f_{n,n}(n) = \frac{1}{(2\pi i)^3}\int_T \frac{F(x,y,t)}{(xyt)^n}\cdot\frac{\tn{d}x\,\tn{d}y\,\tn{d}t}{xyt}, \label{eq:CIF} \end{equation}
where $T$ is any poly-disk defined by $\{|x|=\epsilon_1,|y|=\epsilon_2,|z|=\epsilon_3\}$, with each $\epsilon_j$ sufficiently small.
\end{thm}

The proof follows from the standard Cauchy Integral Formula by induction on the number of variables; see Section 7.2 of Pemantle and Wilson~\cite{PemantleWilson2013} or a standard book on complex analysis in several variables for details.

\section{Proof of Theorem~1}
This section describes how to apply the theory of analytic combinatorics in several variables (ACSV) to the enumeration of centrally weighted GB models.  We focus only on the results of ACSV which we require, and specialize them to our context.  For a general presentation of ACSV we refer the reader to the text of Pemantle and Wilson~\cite{PemantleWilson2013} or the doctoral thesis of Melczer~\cite{Melczer2017}.  To see additional applications of ACSV to lattice path enumeration, see Melczer and Mishna~\cite{MeMi16} or Melczer and Wilson~\cite{MeWi16}.

\subsection{Minimal points}
To find an asymptotic estimate of the coefficients of a univariate generating function using classical methods of (univariate) analytic combinatorics, one typically proceeds by first identifying a finite set of singularities---called the dominant singularities---of the generating function, and then summing asymptotic contributions of each singularity which are determined by their location and nature.  In the study of analytic combinatorics in several variables, the generating function of interest---for us $Q_{\mathfrak a}(1,1;t)$---is represented indirectly as the diagonal of a rational function $F(x,y,t)$; the strategy is then to determine asymptotic behaviour from the analytic properties of $F$.

Recall from Equation~\eqref{eqn:weighted_extraction} that our rational function $F(x,y,t) = \sum_{i,j,n \geq 0} f_{i,j}(n) x^i y^j t^n$ can be written as the ratio of two polynomials $G(x,y,t)$ and $H(x,y,t)$ with 
\[ H(x,y,t) = (1-txyS(\ox,\oy))(1-x)(1-y);\]
for most choices of weights $a$ and $b$ these polynomials  are co-prime.  We define the \emph{singular variety} $\mV$ of $F$ to be the algebraic set composed of its singularities, which is
\[ \mV = \left\{(x,y,t) \in \mathbb{C}^3 : H(x,y,t) = 0 \right\}\] 
when $G$ and $H$ are co-prime.  Let $\mD$ denote the open domain of convergence of the power series representation of $F(x,y,t)$ at the origin.  Similar to the univariate case, where dominant singularities lie on a generating function's disk of convergence, the singularities in the set $\overline{\mD} \cap \mV$, which are called \emph{minimal points}, are crucial to the analysis ($\overline{\mD}$ denotes the closure of $\mD$).  The simple form of the denominator $H$ in the diagonal expression for GB walks makes it easy to identify minimal points.

\begin{lem}
\label{lem:minpts}
For the rational function $F(x,y,t)$ described by Equation~\eqref{eqn:weighted_extraction}, when $G$ and $H$ are co-prime the point $(x,y,t) \in \mV$ is minimal if and only if 
\[|x|\leq  1,\quad |y|\leq 1, \quad |t| \leq \frac{1}{|xy|S(|\ox|,|\oy|)},\]
where these three strict inequalities do not occur simultaneously.
\end{lem}
\begin{proof}
When the polynomials $G$ and $H$ are co-prime the set of minimal points coincides with the minimal points of the rational function $1/H(x,y,t)$, which is the product of three geometric series.  The domain of convergence $\mD$ is then obtained by intersecting the domains of convergence of the rational functions $1/(1-x)$, $1/(1-y)$, and $1/(1-txyS(\ox,\oy))$. The domain of convergence of the last rational function is $\{(x,y,t): |t| < |xy|^{-1}\,S(|\ox|,|\oy|)^{-1}\}$ as the polynomial $xyS(\ox,\oy)$ has non-negative coefficients so that the point $(x,y,t)$ is in its domain of convergence only if $(|x|,|y|,|t|)$ is.
\end{proof}

Similar arguments show that when the weights are such that the factor $(1-x)$, and not $(1-y)$, is canceled by the numerator then the minimal points are described by $|y| \leq 1$ and $|t| \leq \frac{1}{|xy|S(|\ox|,|\oy|)}$; the cases when only $(1-y)$ or both $(1-x)$ and $(1-y)$ are canceled are analogous.

\subsection{Critical points}
\label{ss:critpts}
Intuitively, minimal points are the singularities of $F$ which the cycle of integration $T$ in the Cauchy integral of Equation~\eqref{eq:CIF} can be made arbitrarily close to. The goal is to find a finite set of points such that we can replace the Cauchy integral over $T$ with a sum of integrals over domains local\footnote{Technically, the chains of integration used can be \emph{quasi-local} chains, meaning they are homologous to chains which can be taken to be arbitrarily close to the points we consider, except at points which will not matter in the calculation of dominant asymptotics. A rigorous definition of these quasi-local cycles in the framework of Morse theory is given in Lemma 8.2.4 of Pemantle and Wilson~\cite{PemantleWilson2013}.} to these points. With this in mind, Pemantle and Wilson define \emph{critical points}, which roughly correspond to points where saddle point methods can be applied to local integrals in order to obtain asymptotic estimates.

We will not describe critical points in full generality, instead relying on a proposition which determines critical points for a large class of rational functions (and for our purposes define critical points). First, one must partition the singular variety $\mV$ into a collection of smooth manifolds called \emph{strata}, with certain conditions on how the strata sit inside $\mV$ relative to each other.  Define 
\[ H_1 = 1-txyS(\ox,\oy), \qquad H_2 = 1-x, \qquad H_3 = 1-y,\] 
and for polynomials $f_1,\dots,f_r$ let $\mathbb{V}(f_1,\dots,f_r)$ denote the points $(x,y,t)$ such that $f_1(x,y,t)=\cdots=f_r(x,y,t)=0$. For the rational function we consider, defined in Equation~\eqref{eqn:weighted_extraction}, the set $\mV$ can be stratified into 7 components:
\[ \mV_1 := \mathbb{V}(H_1) \setminus \mathbb{V}(H_2H_3), \quad \mV_2 := \mathbb{V}(H_2) \setminus \mathbb{V}(H_1H_3), \quad \mV_3 := \mathbb{V}(H_3) \setminus \mathbb{V}(H_1H_2), \]
\[ \mV_{12} := \mathbb{V}(H_1,H_2) \setminus \mathbb{V}(H_3), \quad \mV_{13} := \mathbb{V}(H_1,H_3) \setminus \mathbb{V}(H_2), \quad \mV_{23} := \mathbb{V}(H_2,H_3) \setminus \mathbb{V}(H_1),\]
\[ \mV_{123} := \mathbb{V}(H_1,H_2,H_3). \]
Note that this collection of sets partitions $\mV$, and that each defines a smooth manifold (there are additional requirements for this to be a stratification, but they are trivially satisfied).  The following result is sufficient to characterize the critical points for us.

\begin{proposition}[{Pemantle and Wilson~\cite[Section 8.3]{PemantleWilson2013}}]
\label{prop:findcritpts}
Suppose that $X$ is a stratum of $\mV$ defined as $X = \mathbb{V}(f_1,\dots,f_r) \setminus \mathcal{W}$, where $\mathcal{W}$ is an algebraic set of lower dimension than $\mathbb{V}(f_1,\dots,f_r)$, each $\mathbb{V}(f_i)$ is a complex manifold and the tangent planes of the $\mathbb{V}(f_i)$ are linearly independent where they intersect.  Then $(x,y,t) \in X$ is a critical point if and only if the vector $(1,1,1)$ can be written as a linear combination of the vectors 
\[ \left(x\frac{\partial f_i}{\partial x}, y\frac{\partial f_i}{\partial y}, t\frac{\partial f_i}{\partial t}\right), \qquad i=1,\dots,r.\]  
\end{proposition}

\begin{rem}
\label{rem:Lagrange}
This proposition is reminiscent of the classical method of Lagrange multipliers, and we now give a rough idea of this connection. Because $F(x,y,t)$ is absolutely convergent on $(x,y,t) \in \overline{\mD}$,
 the same property holds for $Q_{\mathfrak a}(1,1;xyt)= \sum_{n \geq 0} f_{n,n}(n)\, x^n y^n t^n$. Hence, the radius of convergence $\rho$ of $Q_{\mathfrak a}(1,1;t)$ satisfies
\[ \rho \geq \sup_{(x,y,t) \in \overline{\mD}} \vert x y t\vert .\]
It is not obvious, but true in our context, that this inequality is actually an equality, and the supremum is reached on a point of the variety \mV: 
\begin{equation} \rho = \sup_{(x,y,t) \in \overline{\mD} \cap \mV} \vert x y t\vert .  \label{eq:rho=sup} \end{equation} 
Thus, the finite set of singularities we seek to find corresponds to the maximizers of $\vert x y t\vert $ on $\overline{\mD} \cap \mV$, while the critical points are the candidates for the maximization. Proposition~\ref{prop:findcritpts} encapsulates the fact that the critical points give the possible locations of local extrema of $\vert x y t\vert $ on each stratum. 
\end{rem}

In our application the only polynomial containing the $t$ variable is $H_1$, meaning that the only strata that can have critical points are $\mV_1$, $\mV_{12}$, $\mV_{13}$ and $\mV_{123}$. Using Proposition~\ref{prop:findcritpts} it is then easy to determine the critical points of $\mV$.

\begin{ex} 
On the stratum $\mV_1$, Proposition~\ref{prop:findcritpts} implies that the critical points are those where the matrix 
\[ \begin{pmatrix} x\frac{\partial H_1}{\partial x} & y\frac{\partial H_1}{\partial y} & t\frac{\partial H_1}{\partial t}  \\
1 & 1 & 1 \end{pmatrix} \] 
is rank deficient.  Calculating the $2\times2$ minors of this matrix implies that $(x,y,t)$ is a critical point if and only if 
\[yt(\partial S/\partial x)(\ox,\oy)-xytS(\ox,\oy) = -xytS(\ox,\oy) \text{ and } xt(\partial S/\partial y)(\ox,\oy)-xytS(\ox,\oy) = -xytS(\ox,\oy). \] 
Substituting the definition of the inventory $S(x,y) = ax + (ax)^{-1} + ax(by)^{-1}+ by(ax)^{-1}$ into this system of equations implies there are two critical points defined uniquely by their $(x,y)$-coordinates: $c_1^{\pm} = (\pm a,b)$.  
\end{ex} 

We repeat this computation for each case to determine the following result.  
\begin{proposition}
\label{prop:critpts}
The set of critical points on each strata is given in Table~\ref{tab:critpt}.
\end{proposition}

\begin{table} \centering
\renewcommand{\arraystretch}{1.5}
\begin{tabular}{ccc}
Stratum & Critical Points & Exponential Growth \\\hline
$\mV_1$ & $c_1^{\pm} = (\pm a,b)$ & $e_1 = 4$ \\
$\mV_{12}$ & $c_{12} = \left(1,\frac{b}{a}\right)$ & $e_{12} = \frac{(a+1)^2}{a}$ \\
$\mV_{13}$ & $c_{13}^{\pm} = \left(\pm\frac{a}{\sqrt{b}},1\right)$ & $e_{13} = \frac{2(b+1)}{\sqrt{b}}$ \\
$\mV_{123}$ & $c_{123} = (1,1)$ & $e_{123} = \frac{(b+1)(a^2+b)}{ab}$
\end{tabular}

\caption{The $(x,y)$-coordinates of the critical points; for each, $t=\ox\oy S(\ox,\oy)^{-1}$.}
\label{tab:critpt}
\end{table}

\subsection{Contributing points}
\label{sec:contrib}
The methods of ACSV are well suited to our weighted models because they always a finite set of minimal critical points (which also satisfy some additional technical conditions).  The integral over the cycle $T$ in Equation~\eqref{eq:CIF} can be replaced by the sum of integrals over cycles which are arbitrarily close to each contributing point while introducing an error that grows exponentially smaller than the diagonal sequence. To calculate an asymptotic estimate of the diagonal sequence, we compute the asymptotic contribution for each contributing point and sum the results; the function $\vert x y t\vert^{-1}$ determines the exponential growth of the diagonal sequence.

The next result shows that the function $\vert x y t\vert^{-1}$ is always minimized at a minimal critical point.

\begin{lem}
\label{lem:expmin}
A minimizer $(x_0,y_0,t_0)$ of $\vert x y t\vert^{-1}$ in $\overline{\mD}$ is a minimal critical point.
\end{lem}

\begin{proof}
Assume first that the weights $a$ and $b$ are such that the factors $1-x$ and $1-y$ in $H(x,y,t)$ do not cancel with factors in the numerator $G(x,y,t)$.

By the definition of absolute convergence, the point $(x,y,t) \in \overline \mD$ if and only if $(|x|,|y|,|t|) \in \overline \mD$.  Furthermore, $\vert x y t\vert^{-1}$ decreases as $\vert t \vert$ grows, hence by Lemma~\ref{lem:minpts} the function $\vert x y t\vert^{-1}$ is minimized on $\overline \mD \cap \mathbb R_+^3$ at points of the form $(x,y,\ox\oy S(\ox,\oy)^{-1})$ with $0<x,y \leq 1$.  Thus, it is sufficient to show that the minimizer of 
\[ S(\ox,\oy) = \frac{a}{x}+\frac{ay}{bx}+\frac{bx}{ay}+\frac{x}{a}\] 
for $(x,y) \in (0,1]^2$ occurs at the $(x,y)$-coordinates of a minimal critical point. As touched on in Remark~\ref{rem:Lagrange}, the theory of Lagrange multipliers implies this is true whenever the minimum of $S(\ox,\oy)$ is achieved in $(0,1]^2$. The minimum is achieved because $S(\ox,\oy)$ tends to infinity as either $x$ or $y$ (or both) stay positive and tend to 0: the form of $S(\ox,\oy)$ implies this holds as $x\rightarrow0$ but then it also holds as $y\rightarrow0$ because $x$ bounded away from 0 implies $x/y \rightarrow \infty$.

Similar arguments show that when one or both of the factors $1-x$ and $1-y$ in the denominator are canceled then the result holds as long as the minimum of $S(\ox,\oy)$ is achieved on $(0,1] \times (0,\infty)$, $(0,\infty)\times(0,1]$ or $(0,\infty)^2$, depending on which factors have canceled.  Again the minimum is achieved, as $S(\ox,\oy)$ approaches infinity as either $x$ or $y$ (or both) approach infinity.
\end{proof}

It is now easy to characterize the contributing points using our explicit characterization of minimal points in Lemma~\ref{lem:minpts}.

\begin{proposition}
\label{prop:contrib}
For given weights $a,b>0$, the set of contributing points consists of the unique points $(x,y,t)$ whose $(x,y)$-coordinates are 
\begin{itemize}
\item $c_1^{\pm}$ when $a \leq1$ and $b \leq 1$;
\item $c_{12}$ when $a>1$ and $a\geq b$;
\item $c_{13}^{\pm}$ when $b>1$ and $b\geq a^2$;
\item $c_{123}$ when $b>a>\sqrt{b}>1$.
\end{itemize}
\end{proposition}

Note that on the boundaries of these case distinctions the contributing points with positive coordinates coincide.

\begin{proof}
The result follows from Proposition 10.3.6 of Pemantle and Wilson~\cite{PemantleWilson2013}, which characterizes the set of critical points when $\mV$ has the type of geometry that we encounter here.  Aside from some easily checked technical conditions, we need to show that the points described above in each case are minimal points which (by Lemma~\ref{lem:expmin}) minimize $\vert x y t\vert^{-1}$ among the finite set of minimal critical points on $\overline{\mD}$. 

The values of the exponential growth $\vert x y t\vert^{-1}$ for each set of critical points are listed in the final column of Table~\ref{tab:critpt}.  Note that the AM-GM inequality implies that $e_1 \leq e_{12},e_{13} \leq e_{123}$, so that the set of critical points consists of $c_1^{\pm}$ as long as these points are minimal. Similarly, the points $c_{12}$ or $c_{13}^{\pm}$ contribute as long as they are minimal and $c_1^{\pm}$ is not (note that if $c_1^{\pm}$ is not minimal then it can't happen that both $c_{12}$ and the $c_{13}^{\pm}$ are minimal).  Finally, the conditions listed above come from the characterization of minimal points in Lemma~\ref{lem:minpts}.  

Note that the factors $(1-x)$ and $(1-y)$ in the denominator cancel only when $a=1$ or $b=1$, respectively.  Such models are either Transitional or Directed, and the conclusion can be verified separately for each case.
\end{proof}

At this point we have determined, for all possible choices of weights $a$ and $b$, the points where a local analysis of the rational function $F(x,y,t)$ determines asymptotics of the diagonal sequence up to an exponentially smaller error.  It remains only to calculate the asymptotic contribution of each contributing point and sum the results in each case.  The different strata containing critical points correspond to different exponential growths, but do not completely determine universality classes: the critical exponents depend on several factors, including the degree of vanishing of $G(x,y,t)$ at the contributing points.

We now complete the proof of Theorem~\ref{thm:main_asymptotic_result} by showing how to compute the asymptotic contributions of the critical points.

\subsection{Asymptotic contributions}
\label{ss:asympcontr}

Making the change of variables $(x,y,t) \mapsto \left(\epsilon_1 e^{i\theta_1}, \epsilon_2 e^{i\theta_2}, \epsilon_3 e^{i\theta_3}\right)$ converts the Cauchy Integral representation~\eqref{eq:CIF} into
\begin{equation} f_{n,n}(n) = \frac{(\epsilon_1\epsilon_2\epsilon_3)^{-n}}{(2\pi)^3} \int_{(-\pi,\pi)^3} F\left(\epsilon_1 e^{i\theta_1}, \epsilon_2 e^{i\theta_2}, \epsilon_3 e^{i\theta_3} \right) e^{-n(\theta_1+\theta_2+\theta_3)} \tn{d}\theta_1 \tn{d}\theta_2 \tn{d}\theta_3. \label{eq:CIFexp} \end{equation}
This new representation is an example of a \emph{Fourier-Laplace integral}: that is, it has the form 
\[ \int A(\theta_1,\theta_2,\theta_3) e^{-n \phi(\theta_1,\theta_2,\theta_3)} \tn{d}\theta_1 \tn{d}\theta_2 \tn{d}\theta_3, \]
with the functions $A$ and $\phi$ analytic over their domain of integration.  Asymptotics of a large class of Fourier-Laplace integrals can be obtained through the following result.

\begin{proposition}[{H{\"o}rmander~\cite[Theorem 7.7.5]{Hormander1990a} \& Pemantle and Wilson~\cite[Lemma 13.3.2]{PemantleWilson2013}}] 
\label{prop:HighAsm}
Suppose that the functions $A(\bt)$ and $\phi(\bt)$ in $d$ variables are smooth in a neighbourhood $\mN$ of the origin and that
\begin{itemize}
\item $\phi$ has a critical point at $\bt=\bzer$, i.e., that $(\nabla \phi)(\bzer)=\bzer$, and that the origin is the only critical point of $\phi$ in $\mN$;
\item the Hessian $\mH$ of $\phi$ at $\bt=\bzer$ is non-singular;
\item $\phi(\bzer)=0$;
\item the real part of $\phi(\bt)$ is non-negative on $\mN$.
\end{itemize} 
Then for any integer $M>0$ there exist effective constants $C_0,\dots,C_M$ such that
\begin{equation}
\label{eq:Goal}
 \int_{\mN} A(\bt)\,e^{-n \phi(\bt)}\tn{d}\bt = \left(\frac{2\pi}{n}\right)^{d/2} \det(\mH)^{-1/2} \cdot \sum_{k=0}^M C_k n^{-k} + O\left(n^{-M-1}\right). 
\end{equation}
The constant $C_0$ is equal to $A(\bzer)$. Moreover, if $A(\bt)$ vanishes to order $L$ at the origin then (at least) the constants $C_0,\dots, C_{\lfloor\frac{L}{2}\rfloor}$ are all zero.
\end{proposition}

The constants $C_k$ are determined by evaluating the application of powers of an explicit differential operator to a sequence of explicit functions depending on $\phi$ and $A$.  As the formulas get quite involved, we refer the reader to H{\"o}rmander~\cite[Theorem 7.7.5]{Hormander1990a} or Pemantle and Wilson~\cite[Lemma 13.3.2]{PemantleWilson2013}. For our results we calculated the constants by writing a simple program in the Maple computer algebra system\footnote{The computations are available for download at at \url{http://weighted-walks.gforge.inria.fr}}. 

Unfortunately we cannot directly apply Proposition~\ref{prop:HighAsm} to Equation~\eqref{eq:CIFexp} as the function $\phi=\theta_1+\theta_2+\theta_3$ admits no critical points (in the usual calculus sense).  The importance of critical points in the context of ACSV comes from the fact that localizing the Cauchy Integral around the critical points results in an integral expression---or sum of expressions---to which Proposition~\ref{prop:HighAsm} can be applied.  We now see how this can be achieved for the different cases laid out in Proposition~\ref{prop:contrib}.

\subsubsection{The Balanced Case}
This is the unweighted model, whose asymptotics have been previously computed~\cite{BoMi10}. Still, the solution via ACSV initiates us to the weighted models, so we present it here.

When $a=b=1$, the rational function under consideration simplifies to 
\[ F(x,y,t) = \frac{yt^2(1+x)(y-x^2)(x-y)(y+x)}{1-txyS(\ox,\oy)}, \]
although we keep $a$ and $b$ as symbolic parameters below as this helps illustrate the form of the general argument.  From Proposition~\ref{prop:contrib} we know that the growth of the coefficient is determined by the singular behaviour of $F$ near the points $c_1^{\pm} = \left(\pm a,b, \pm \oa\ob S(\oa,\ob)^{-1}\right)$, and we derive this again in the course of estimating integrals.  Consider the chains of integration
\begin{align*} 
T_- &= \{ (x,y,t) : |x| = a, |y| = b, |t| = \oa\ob S(\oa,\ob)^{-1} - \epsilon \},\\
T_+ &= \{ (x,y,t) : |x| = a, |y| = b, |t| = \oa\ob S(\oa,\ob)^{-1} +\epsilon \},
\end{align*}
for $\epsilon$ arbitrarily small.  By the Cauchy Integral Formula and minimality of $c_1^{\pm}$ (see Lemma~\ref{lem:minpts}),
\[ f_{n,n}(n) = \frac{1}{(2\pi i)^3} \int_{T_-} \frac{F(x,y,t)}{(xyt)^n}\,\frac{\tn{d}x\,\tn{d}y\,\tn{d}t}{xyt}. \]
As every point on $T_+$ lies outside the domain of convergence of the generating function $F(x,y,t)$, one can show that 
\[ \frac{1}{(2\pi i)^3} \int_{T_+} \frac{F(x,y,t)}{(xyt)^n}\,\frac{\tn{d}x\,\tn{d}y\,\tn{d}t}{xyt} \]
grows exponentially smaller than the diagonal sequence as $n$ approaches infinity.  Defining
\[ D_{\pm} := \{z : |z| = \oa\ob S(\oa,\ob)^{-1} \pm \epsilon\}, \]
it follows that
{\footnotesize 
\begin{align*} 
f_{n,n}(n)  = \frac{1}{(2\pi i)^3} \int_{|x|=a}\int_{|y|=b} &\left( \int_{D_+} \frac{yt^2(1+x)(y-x^2)(x-y)(y+x)}{1-txyS(\ox,\oy)}\frac{\tn{d}t}{(xyt)^{(n+1)}} \right. \\
& \qquad \left. -\int_{D_-} \frac{yt^2(1+x)(y-x^2)(x-y)(y+x)}{1-txyS(\ox,\oy)}\frac{\tn{d}t}{(xyt)^{(n+1)}} \right)\, \tn{d}x\, \tn{d}y + O(K^n)
\end{align*}
}

\noindent
for some constant $K>0$ with $K < \limsup_{n\rightarrow\infty}|f_{n,n}(n)|^{1/n}$.  The inner difference of integrals can then be determined by a residue computation.  

By design, for $\epsilon$ sufficiently small, dominant asymptotics only depend on integration around arbitrarily small neighbourhoods $\mN_+$ and $\mN_-$ of the critical points $c_1^{\pm}$ in the set $\{(x,y):|x|=a,|y|=b\}$, so
{\small
\begin{align}
f_{n,n}(n)  &= \frac{1}{(2\pi i)^2} \sum_{j \in \{+,-\}} \int_{\mN_j} \ox^2\oy S(\ox,\oy)^{-2}(1+x)(y-x^2)(x-y)(y+x)S(\ox,\oy)^n \tn{d}x \tn{d}y + O(K^n) \nonumber \\
&= \frac{S(\oa,\ob)^n}{(2\pi)^2} \sum_{j \in \{+,-\}} \int_{(-\epsilon',\epsilon')^2} A_j(\theta_1,\theta_2) e^{-n \phi_j(\theta_1,\theta_2)} + O(K^n), \label{eq:smooth_ints} 
\end{align}
}
with 
\begin{align*}
A_j(\theta_1,\theta_2) &= G\left(ja e^{i\theta_1}, b r_2e^{i\theta_2} \right), \\
\phi_j(\theta_1,\theta_2) &= -\log  S\left(j \oa e^{-i\theta_1},\ob e^{-i \theta_2}\right) + \log  S\left(\oa,\ob\right).
\end{align*} 
By the definition of minimal critical points, the conditions of Proposition~\ref{prop:HighAsm} are satisfied by these integrals at $c_1^{\pm}$, and applying the proposition gives the asymptotic estimate
\[ f_{n,n}(n) = \frac{8}{\pi} \cdot \frac{4^n}{n^2} \left(1 + O\left(\frac{1}{n}\right)\right). \]

\subsubsection{The Reluctant Case}

The integral expression for the reluctant case is roughly the same as in the balanced case, except that now the factors $(1-x)$ and $(1-y)$, which canceled with factors in the numerator when $a=b=1$, appear in the denominator of $F(x,y,t)$.  In this case, when $a<1$ and $b<1$, Proposition~\ref{prop:contrib} shows that the contributing points have $x$- and $y$-coordinates of modulus less than~$1$.  The argument for the balanced case can then be applied to obtain an expression of the same form as Equation~\eqref{eq:smooth_ints}, since the domain of integration avoids the zero set of $(1-x)(1-y)$.  In particular, we can apply Proposition~\ref{prop:HighAsm} to the integrals in Equation~\eqref{eq:smooth_ints} where now 
\[ A_j = \frac{G\left(ja e^{i\theta_1}, b r_2e^{i\theta_2} \right)}{(1-ja e^{i\theta_1})(1-b r_2e^{i\theta_2})}, \] and the $\phi_j$ are unchanged from the balanced case.  Shifting the asymptotics to take the relationship in Equation~\eqref{eqn:weighted_extraction} between the diagonal sequence and the number of lattice walks into account gives the asymptotics listed in Theorem~\ref{thm:main_asymptotic_result}.

\subsubsection{The Transitional Cases}

The transitional cases are on the boundary between being reluctant and directed.  When $a=1$ and $b<1$, the contributing points $c_1^{\pm}$ have $x$-coordinate of modulus $1$, however the factor of $1-x$ cancels with a factor of $1-x$ which becomes present in the numerator when specializing $a$ to $1$.  When $b=1$ and $a<1$, the contributing points $c_1^{\pm}$ have a $y$-coordinate of $1$, however the factor of $1-y$ cancels with a factor of $1-y$ which becomes present in the numerator when specializing $b$ to $1$.  After this simplification, the same argument as in the balanced and reluctant cases applies.

Note also that this cancellation shows why the balanced, transitional and reluctant cases have the same exponential growth but critical exponents $\alpha$ differing by 1.  The contributing points are the same for each, but the order of vanishing of the numerator $G$ at the critical points is $2$, $3$ and $4$ for balanced, transitional and reluctant, respectively.  Proposition~\ref{prop:HighAsm} shows that when the order of vanishing of the numerators increases one expects the critical exponents to decrease.

\subsubsection{The Directed Cases}

In the directed cases, the contributing points lie on the intersection of the zero set of two factors in the denominator of $F(x,y,t)$.  This requires a slightly different argument, but the general approach is the same.

Suppose first that we are in the case when $a>1$ and $a>b$, so that the unique contributing point has $(x,y)$-coordinates $c_{12} = (1,b/a)$ and lies on $\mV_{12}$.  If we define
\[ \widetilde{G}(x,y) := \frac{G(x,y,t)}{t^2} = y(y-b)(a-x)(a+x)(a^2y-bx^2)(ay-bx)(ay+bx)  \]
then the Cauchy Integral Formula implies that we want asymptotics of
{\small
\begin{align*} 
[t^{n+2}][x^{n+2}][y^{n+2}] \left( \frac{t^2\widetilde{G}(x,y)}{a^4b^3(1-txyS(\ox,\oy))(1-x)(1-y)} \right)  
&= [x^2][y^2] \left( \frac{\widetilde{G}(x,y)S(\ox,\oy)^n}{a^4b^3(1-x)(1-y)} \right) \\[+2mm]
&= \underbrace{\frac{1}{a^4b^3(2 \pi i)^2} \int_{|y|=b/a} \int_{|x|=1-\epsilon} \frac{\widetilde{G}(x,y) S(\ox,\oy)^n}{(1-x)(1-y)} \frac{dxdy}{x^3y^3}}_I,
\end{align*}
}

\noindent
for any fixed $0 < \epsilon < 1$. Furthermore, standard integral bounds imply the existence of constants $C>0$ and $M \in (0,2+a+1/a)$ such that 
\begin{align*} 
\left|  \frac{1}{a^4b^3(2 \pi i)^2} \int_{|y|=b/a} \int_{|x|=1+\epsilon} \frac{\widetilde{G}(x,y) S(\ox,\oy)^n}{(1-x)(1-y)} \frac{dxdy}{x^3y^3} \right|
&\leq C \cdot M^n.
\end{align*}
Thus, as $n \rightarrow \infty$,
\begin{align*} 
I = \frac{1}{a^4b^3(2 \pi i)^2} \int_{|y|=b/a} & \left(\int_{|x|=1-\epsilon} \frac{\widetilde{G}(x,y) S(\ox,\oy)^n}{(1-x)(1-y)} \frac{dxdy}{x^3y^3} \right. \\
&\left. \qquad\qquad\qquad - \int_{|x|=1+\epsilon} \frac{\widetilde{G}(x,y) S(\ox,\oy)^n}{(1-x)(1-y)} \frac{dxdy}{x^3y^3} \right)dy + O(M^n).
\end{align*}
Cauchy's residue theorem then implies
\begin{align} 
I &= \frac{1}{a^4b^3(2 \pi i)} \int_{|y|=b/a} \frac{\widetilde{G}(1,y) S(1,\oy)^n}{1-y} \frac{dy}{y^3} + O(M^n)\notag \\[+2mm] 
&= \frac{(2+a+1/a)^n}{a^4b^3(2 \pi)} \int_{-\pi}^\pi A(\theta) e^{-k \phi(\theta)} d\theta + O(M^n), \label{eq:DirFL}
\end{align}
where 
\[ A(\theta) = \frac{\widetilde{G}(1,(b/a)e^{i\theta})}{(b/a)^2e^{2i\theta}(1-(b/a)e^{i\theta})} \qquad \text{and} \qquad \phi(x,y) = \log S(1,a/b) -\log S\left(1,(a/b)e^{-i\theta}\right). \]

Since $c_{12}$ is a contributing singularity, it can be shown that the integral in Equation~\eqref{eq:DirFL} has a single critical point at $\theta =0$, to which Proposition~\ref{prop:HighAsm} (which is classical in the univariate case) can be applied.  In the other directed case, when $b>1$ and $\sqrt{b}>a$, the analysis is the same except that the existence of two contributing points, with $(x,y)$-coordinates $c_{13}^{\pm}$, implies that asymptotics are ultimately obtained from a sum of two univariate Fourier-Laplace integrals.

\paragraph{Comparison to the Method of Multivariate Residues}

Above we have adapted the tools of saddle-point analysis in order to determine asymptotics when our contributing points lie on a strata defined by the intersection of two smooth manifolds.  The theory of analytic combinatorics in several variables, as developed in Pemantle and Wilson~\cite{PemantleWilson2013}, takes a more uniform (but less elementary) approach to such integrals using the theory of multivariate residues.  The theoretical groundwork for multivariate complex residues dates back to Leray, and is well presented in the text of A{\u\i}zenberg and Yuzhakov~\cite{AuizenbergYuzhakov1983}.  Instead of detailing the theory here we refer the reader to Section~10.2 of Pemantle and Wilson~\cite{PemantleWilson2013}, which describes how to calculate the multivariate residue of a rational function over a fixed strata.  

Consider again the case when $a>1$ and $a>b$, so that the unique contributing point has $(x,y)$-coordinates $c_{12} = (1,b/a)$ and lies on $\mV_{12}$.  As $c_{12}$ defines the unique contributing point, Theorem 10.2.2 of Pemantle and Wilson~\cite{PemantleWilson2013} implies that diagonal asymptotics are determined by the integral 
\[ \frac{1}{2\pi i} \int_{\sigma} \frac{\widetilde{G}(1,y) S(1,\oy)^n}{1-y} \frac{dy}{y^3}, \]
where $\sigma$ is a one-dimensional curve of integration in $\mV_{12}$ containing the contributing point.  Unfortunately, it can be difficult in general to obtain an explicit description of the possible domains of integration $\sigma$ using this approach.  When the numerator $G$ of the rational function under consideration does not vanish at a contributing point, Pemantle and Wilson~\cite[Theorem 10.3.4]{PemantleWilson2013} give an explicit formula for its dominant asymptotic contribution using the multivariate residue approach.

\subsubsection{The Free Case}
In the free case $\sqrt{b}<a<b$ there is exactly one contributing point, whose $(x,y)$-coordinates are $(1,1)$.  This point lies on the strata determined by the intersection of the three varieties $\mV_1,\mV_2,\mV_3$; when the co-dimension of a strata equals the dimension of the ambient space, as it does here, the intersection of varieties is called a \emph{complete intersection}.  Complete intersections have an interesting property: the three dimensional (multivariate) residue, when defined, is a constant!  This means that for a complete intersection one can determine dominant asymptotics up to an exponentially smaller error term (instead of a polynomial error term as in most cases above).

Indeed, Theorem 10.3.1 of Pemantle and Wilson~\cite{PemantleWilson2013} explicitly gives the dominant asymptotics when the contributing point lies on a complete intersection and the numerator $G$ does not vanish at this point.  This faster, exponential, decrease in the relative error can be observed numerically while examining terms of modest size (less than 100).

\subsubsection{The Axial Cases}
The axial cases are on the boundary of the directed cases and the free case.  Unfortunately, the above approaches cannot be directly used for these cases as the numerator $G$ vanishes at the contributing points and $\mV$ is locally a complete intersection at these points. Luckily, we can decompose the rational function under consideration into two simpler rational functions and analyze them.

Consider the case when $a=b>1$, so that we have the generating function relationship $Q_{\mathfrak a}(1,1;t) = \frac{1}{a^4t^2} \Delta F(x,y,t)$ with 
\[ F(x,y,t) = \frac{yt^2(y-a)(a-x)(a+x)(ay-x^2)(y-x)(y+x)}{(1-txyS(\ox,\oy))(1-x)(1-y)}, \]
and contributing point $(1,1,1/S(1,1))$.  Because we cannot analyze the above diagonal we use a trick, also employed by Melczer and Wilson~\cite{MeWi16}, to write $y-x = (1-x)-(1-y)$ and see that
{\small \[ F(x,y,t) =  \underbrace{\frac{yt^2(y-a)(a-x)(a+x)(ay-x^2)(y+x)}{(1-txyS(\ox,\oy))(1-y)}}_{F_1(x,y,t)} - \underbrace{\frac{yt^2(y-a)(a-x)(a+x)(ay-x^2)(y+x)}{(1-txyS(\ox,\oy))(1-x)}}_{F_2(x,y,t)}. \]
}
As the diagonal operator is linear, we can obtain the desired walk asymptotics by studying $\Delta F_1$ and $\Delta F_2$.  Applying the same arguments as Section~\ref{sec:contrib} shows that $F_1$ admits contributing points with $(x,y)$-coordinates $c_{13}^{\pm}=(\pm\sqrt{a},1)$ while $F_2$ admits the contributing point with $(x,y)$-coordinates $c_{12}=c_{123}$.  The diagonal of $F_2$ has larger exponential growth, and its numerator does not vanish at the contributing point, meaning Theorem 10.3.4 of Pemantle and Wilson~\cite{PemantleWilson2013} can be used to determine dominant asymptotics.

When $b=a^2$ the argument is analogous except the numerator contains $y-x^2 = (1-x)(1+x)-(1-y)$ as a factor and this is used to decompose the diagonal into a sum of two simpler diagonals, which are then analyzed.

\subsection{Excursion asymptotics}
The diagonal expression in Equation~\eqref{eqn:weighted_extraction} was obtained through a power series extraction in Equation~\eqref{eq:Qpospart}.  The reason for the factors $(1-x)$ and $(1-y)$ in the denominator of the expression to convert the series extraction into a diagonal is to account for terms corresponding to all non-negative powers of $x$ and $y$.  Thus, the generating function~$Q_{\mathfrak a}(0,0;t)=\sum_{n\geq0} e_{(0,0)\rightarrow (0,0)} t^n$ for walks beginning and ending at the origin can be expressed as the diagonal
\begin{equation}
\label{eqn:exc_weighted_extraction_nonan}
Q_{\mathfrak a}(0,0;t) = \Delta \left( \frac{(y-b)(a-x)(a+x)(a^2y-bx^2)(ay-bx)(ay+bx)}{a^4b^3x^2y(1-txyS(\ox,\oy))} \right),
\end{equation}
which is the diagonal in Equation~\eqref{eqn:weighted_extraction_nonan} except for a factor of $(1-x)(1-y)$ in the denominator.  This means that excursion asymptotics are easy to calculate: the singular variety is now smooth and the smooth critical points with $(x,y)$-coordinates $c_1^{\pm} = (\pm a,b)$ are always contributing for any non-negative choice of weights $a$ and $b$.  The arguments detailed above then yield the following result.

\begin{thm}\label{thm:excursion_asympt}
For any non-negative weights $a, b>0$, the number of excursions $e_{(i,j)\rightarrow(0,0)}(n)$ has dominant asymptotics
\[ e_{(i,j)\rightarrow(0,0)}(n) =
\begin{cases}
\frac{4^n}{n^5} \left(\frac{128(j+1)(1+i)(3+i+2j)(2+i+j)}{a^ib^j\pi} + O\left(\frac{1}{n}\right)\right) &\text{if } n+i \equiv 0 \pmod{2}, \\
0 &\text{if } n+i \equiv 1 \pmod{2}.
\end{cases}\]
\end{thm}

By suitably modifying the diagonal expression one can also calculate asymptotics for walks beginning at any point and ending at any point or axis using the above arguments.   

\subsection{General ending point}
As mentioned above, the denominator $H(x,y,t)$ of the diagonal expression obtained for the generating function $Q_{\mathfrak a}^{i,j}(x,y;t)$ for walks ending at $(i,j)$ is the same as the denominator for the walks ending at the origin.  This means that all the above results about minimal and critical points continue to hold when $i$ and $j$ are included as parameters.  In fact, the degree of vanishing of the numerator at each contributing point is independent of the starting point $(i,j)$, which is why the critical exponents $\alpha$ depend only on the weights $a$ and $b$.  Repeating the above analysis with the numerator $G(x,y)$ replaced by the new expression parametrized by $i$ and $j$ yields the formulas in Appendix~\ref{sec:values_harmonic_function}.

\part{General considerations for weighted models}
\label{part:CentralWeights}
Restricting ourselves to central weightings results in models that are well suited to an analysis through generating functions, and we begin our discussions in this section by providing different characterizations of central weightings.  The characterizations we derive have significant consequences, and we illustrate several of them using our results on GB walks in Section~\ref{sec:Diagrams}. We conclude with some remarks on related properties of the generating functions under consideration.

\section{Central weightings}
\label{sec:central}

\subsection{Characterizing central weights}
\label{sec:characterize}
In the first part of this article we considered the GB model under the particular assignment of weights in Equation~\eqref{eq:GB_weighting}. This section gives insights on general weighted models, illustrating our choice of these weights.  Here we consider step sets $\mS \subset \mathbb{Z}^d$ in arbitrary dimension $d$ and given $s \in \mS$ we denote the $k$th coordinate of $s$ by $\pi_k(s)$.   Weights assigned to steps will always be positive.  Recall the following definition of central weights.

\begin{defn}
A weighting is \emph{central} if all paths with the same length, start and end points have the same weight.
\label{def:central}
\end{defn}

We give other characterizations of this definition, and see that this is a good framework in which to study weighted models.  It leads to natural generating function relations, and permits us to study the asymptotic behaviour in a unified way. A step set $\mS \subset \mathbb{Z}^d$ is called \emph{non-singular} if it is not contained in a half-space of $\mathbb{R}^d$.

\begin{thm}
\label{thm:centralweighting}
Let $\mS \subset \mathbb{Z}^d$ be a set of non-singular integer steps of dimension $d$. A weighting is central if and only if any of the following equivalent statements holds:
\begin{enumerate}
\item\label{(i)}For every point $(i_1,\dots,i_d)$ and number $n$, each walk starting at the origin and ending at $(i_1,\dots,i_d)$ in $n$ steps, while staying in $\mathbb R_{\geq 0}^d$, has the same weight;
\item\label{(ii)}There exist constants $\alpha_1,\dots,\alpha_d$ and $\beta$ such that the weight assigned to each step $s \in \mS$ is of the form $a_s = \beta \prod_{k=1}^d \alpha_k^{\pi_k(s)}$;
\item\label{(iii)}There exist constants $\alpha_1,\dots,\alpha_d$ and $\beta$ such that
\[K_{\mathfrak a}(x_1,\dots,x_d;t) = K(\alpha_1 x_1,\dots, \alpha_d x_d;\beta t), \]
where $K_{\mathfrak a}(x_1,\dots,x_d;t)$ denotes the kernel for a model of walks with a weight $a_s$ for each step $s \in \mS$, and $K(x_1,\dots,x_d;t)$  the kernel for the unweighted model:
\[K_{\mathfrak a}(x_1,\dots,x_d;t) = 1 -  t\,\sum_{s \in \mS} a_s \prod_{i=1}^d x_k^{\pi_k(s)}, \quad K(x_1,\dots,x_d;t) = 1 - t \,\sum_{s \in \mS}  \prod_{i=1}^d x_k^{\pi_k(s)}.\]
\end{enumerate}
\end{thm}

Although most of the induced implications are quite straightforward to show, the proof of this theorem will be postponed to the next subsection; the implications $\ref{(i)} \Rightarrow \ref{(ii)}$ and $\ref{(i)} \Rightarrow \ref{(iii)}$ are quite delicate to prove and require an intermediate definition which is quite long to state (see  Proposition~\ref{prop:productweights}).  In probability theory, the weights $a_s$ given in Condition~\ref{(ii)} constitute an exponential change (sometimes called a Cram\'er transform) of the uniform weights on $\mathcal S$.

\begin{example} 
For the Gouyou-Beauchamps model with weight assignment \eqref{eq:GB_weighting}, we recover the definition of central weighting via the substitutions $\alpha_1=a$, $\alpha_2 = b$ and $\beta = 1$. We have lost a degree of freedom by setting $\beta = 1$ but it can be simply restored by multiplying the numbers of walks of length $n$ by $\beta^n$. 
\end{example}

Note that Condition~\ref{(ii)} can be expressed in a matrix form. Let $M_{\mS}$ be the matrix
\begin{equation}
M_{\mS}=
\begin{pmatrix}
\pi_1(s_1) & \pi_2(s_1) & \cdots & \pi_d(s_1) & 1 \\
\pi_1(s_2) & \pi_2(s_2) & \cdots & \pi_d(s_2) & 1 \\
\ \ \ \vdots & \ \ \ \vdots & \ & \ \ \ \vdots & \vdots\\
\pi_1(s_m) & \pi_2(s_m) & \cdots & \pi_d(s_m) & 1
\end{pmatrix},
\label{eq:MatrixS}
\end{equation}
where $\mS = \{s_1,s_2,\dots,s_m\}$.  Then a weighting $(a_s)_{s \in \mS}$ satisfies Condition~\ref{(ii)} if and only if there exist constants $\alpha_1,\dots,\alpha_d$ and $\beta$ such that
\[
\begin{pmatrix} \log(a_{s_1}) \\ \log(a_{s_2}) \\\ \ \ \vdots \\ \log(a_{s_m}) \end{pmatrix}  =  
M_{\mS}  \begin{pmatrix} \log(\alpha_1) \\ \ \ \ \vdots \\ \log(\alpha_d) \\ \log(\beta) \end{pmatrix}. 
\]

\begin{lem} 
\label{lem:rank}
If $\mS$ is a set of non-singular integer steps of dimension $d$, then the rank of the matrix $M_{\mS}$ defined by \eqref{eq:MatrixS} is $d+1$.
\end{lem}

\begin{proof} 
Let us assume that the rank of $M_{\mS}$ is at most $d$: we are going to prove that $\mS$ is contained in a half-space (and thus $\mS$ defines a singular model).  Since the rank of $M_{\mS}$ is at most $d$, and any non-singular model has at least $d+1$ steps, there exist $d$ steps $t_1,\dots,t_d$ such that the family  $(\pi_1(t_j),\dots,\pi_d(t_j),1)$ indexed by ${j \in \left\{1,\dots d\right\}}$  spans every other vector $(\pi_1(s),\dots,\pi_d(s),1)$ with $s \in \mS$. In other words, every step $s \in \mS$ belongs to
\[ A := \left\{ \left. \sum_{j=1}^d q_j \, t_j \ \ \right| \ \
    (q_1,\dots,q_d) \in \mathbb{Q}^d \textrm{ with } \sum_{j=1}^d q_j
    = 1 \right\}. \]
The set $A$ is an affine hyperplane contained in the linear span of
$\left\{t_j\right\}_{j \in \{1,\dots,d\}}$, so it is an affine subspace
of $\mathbb R^d$ with dimension at most $d - 1$. Therefore
$A$, and a fortiori~$\mS$, is contained in a half-space so the model
is singular.

Thus, since $M_{\mS}$ has $d+1$ columns, if $\mS$ is non-singular then the rank of $M_{\mS}$ equals $d+1$.
\end{proof}

\subsection{A combinatorial interpretation of the relations}
\label{sec:paths}
We now combinatorially interpret the result of Lemma~\ref{lem:rank} in terms of paths.

\begin{proposition} 
\label{prop:p1p2}
Given a non-singular integer step set~$\mS$, there exists a subset $\mathcal T$ of $\mS$ with $d+1$ steps such that for every~$s \in \mS \setminus \mathcal T$ there are two paths $p_s$ and~$p'_s$ in~$\mathbb R^d$ (not necessarily $\mathbb{R}_{\geq0}^d$) such that:
\begin{itemize}
\item $p_s$ and $p'_s$ begin at the origin, and have the same length and the same endpoint;
\item $p_s$ contains $s$ as a step, with all its other steps belonging to $\mathcal T$;
\item $p'_s$ only uses steps in $\mathcal T$.
\end{itemize}
\end{proposition}

\begin{example} 
	Consider the Gouyou-Beauchamps model with the set $\mathcal T = \{(1,0),(-1,0),(1,-1)\}$. For $s = (-1,1)$ we can choose $p_s$ to be the concatenation of $(-1,1)$ and $(1,-1)$, and $p'_s$ to be the concatenation of $(1,0)$ and $(-1,0)$. 
\end{example}

\begin{proof} 
By Lemma~\ref{lem:rank}, the rank of $M_{\mS}$ is $d+1$, hence we can find a set $\mathcal T$ of $d+1$ steps such that $(\pi_1(t),\dots,\pi_d(t),1)_{t \in \mathcal T}$ spans every other vector $(\pi_1(s),\dots,\pi_d(s),1)$. Given a step $s \in \mS \setminus \mathcal T$, the vector $(\pi_1(s),\dots,\pi_d(s),1)$ is then a linear combination of $(\pi_1(t),\dots,\pi_d(t),1)_{t \in \mathcal T}$ with rational coefficients. Thus, we can multiply this linear combination by the coefficients' common denominator to obtain a relation with only integer numbers. Reorganizing terms according to their signs induces a relation of the form
{\small 
\begin{equation}
n_s\,(\pi_1(s),\dots,\pi_d(s),1) + \sum_{j=1}^{\ell} n_{s,t_j} (\pi_1(t_j),\dots,\pi_d(t_j),1) = \sum_{j=1}^{\ell'} n'_{s,t'_j} (\pi_1(t'_j),\dots,\pi_d(t'_j),1),
\label{eq:linearcombi}
\end{equation} 
}
where $t_1,\dots,t_{\ell},t'_{1},\dots,t'_{\ell'}$ are steps of $\mathcal T$ and $n_s,n_{s,t_1},\dots,n_{s,t_{\ell}},n'_{s,t'_1},\dots,n'_{s,t'_{\ell'}}$ are positive integers. We prove the proposition by considering $p_s$ as any path formed by $n_s$ steps of the type $s$ and $n_{s,t_j}$ steps of the type $t_j$, and $p'_s$ as any path formed by $n_{s,t'_j}$ steps of the type $t'_j$.
\end{proof}

We can now give a last equivalent definition of central weighting.

\begin{proposition} 
\label{prop:productweights}
Consider a non-singular integer step set~$\mS$ in $\mathbb R^d$, along with a set $\mathcal T$ and the corresponding $|\mS| -|\mathcal T| = |\mS|-d-1$ pairs of paths $(p_s,p'_s)_{s \in \mS \setminus \mathcal T}$ described in Proposition~\ref{prop:p1p2}.

A weighting is central if and only if 
\begin{enumerate}
\setcounter{enumi}{3}
\item\label{(iv)} for every $s \in \mS \setminus \mathcal T$, the weights $(a_r)_{r \in \mS}$ satisfy
\begin{equation}
 \prod_{r \in p_s} a_r = \prod_{r' \in p'_s} a_{r'}, \label{eq:prodas}
\end{equation}
where the steps are considered with multiplicity inside each product.
\end{enumerate}
\end{proposition}

\begin{example}
We continue the GB example with the aforementioned set $\mathcal T$, and paths $p_s,p_s'$;  Proposition~\ref{prop:productweights} gives
\[ a_{-1,1} \times a_{1,-1} = a_{1,0} \times a_{-1,0}. \]
In fact, this relation does not depend (after potential simplification) on the choice of $\mathcal T$, $p_s$, or $p_s'$.
\end{example}

We can now provide a combined proof of Theorem~\ref{thm:centralweighting} and Proposition~\ref{prop:productweights}.

\begin{proof}[Proof of Theorem~\ref{thm:centralweighting} and Proposition~\ref{prop:productweights}]\mbox{}
\begin{description}
\item{$\ref{(ii)} \boldsymbol \Leftrightarrow \ref{(iii)}$.}
This equivalence translates an identification of coefficients between polynomials.

\item{Definition~\ref{def:central} $ \boldsymbol \Rightarrow \ref{(i)}$.}
Obvious: we constrain the starting point to be the origin.

\item{$\ref{(i)} \boldsymbol \Rightarrow$ \ref{(iv)}.}
Since $\mS$ is non-singular, we can find a path starting from the origin and ending at a point arbitrarily far from both the $x$- and $y$-axes. Thus, we can pick some path $p$ using the steps $\mS$ such that for every $s \in \mS \setminus \mathcal T$ the concatenation of $p$ and $p_s$ and the concatenation of $p$ and $p'_s$ stay in $\mathbb R^d_{\geq0}$. The weighting being central, the weights of these two walks are equal, giving Equation~\eqref{eq:prodas} after some simplification.

\item{$\ref{(iv)} \boldsymbol \Rightarrow \ref{(ii)}$.}
Assume that Equation~\eqref{eq:prodas} holds. We will prove that the image of the matrix $M_{\mS}$ defined by Equation~\eqref{eq:MatrixS} is equal to the set
\begin{equation*}
E = \left\{ (y_s)_{s \in \mS} \ \left| \ \ \forall s \, \in \, \mS \setminus \mathcal T, \ \ \sum_{r \in p_s} y_r  = \sum_{r' \in p'_s} y_{r'} \textrm{ (considered with multiplicity)} \right. \right\}.
\end{equation*}
A vector $\left(y_s\right)_{s \in \mS}$ in $\textrm{Im}(M_{\mS})$ can be parametrized as
\[y_s = x_{d+1} + \sum_{i=1}^d \pi_i(s)x_i\] 
for $s \in \mS$ and indeterminates $x_i$. For any path $q$ the coefficient of $x_{d+1}$ in the sum $\sum_{r \in q} y_r$ is the length of $q$, while the coefficient of $x_k$ for $1 \leq k \leq d$ is the $k$th coordinate of the endpoint of $q$.  As the paths $p_s$ and $p'_s$ coincide at their endpoints and have the same length for every $s \in \mS \setminus \mathcal T$, we see that $(y_s)$ belongs to $E$. In other words, $\textrm{Im}\left(M_{\mS}\right) \subseteq E$. 

The equality $\textrm{Im}\left(M_{\mS}\right) = E$ follows from considering the dimensions of these linear spaces. On one hand, the dimension of $\textrm{Im}\left(M_{\mS}\right)$ is $d+1$ by Lemma~\ref{lem:rank}. On the other hand, the dimension of $E$ is also $d+1$ since $E$ is the intersection of $|S|-d-1$ disjoint hyperplanes: the hyperplanes of equations $\sum_{r \in p_s} y_r  = \sum_{r' \in p'_s} y_{r'}$ for $s  \in  \mS \setminus \mathcal T$. (These hyperplanes are disjoint because   $\sum_{r \in p_s} y_r$ is the only sum involving the coordinate $y_s$, as per the second condition of Proposition~\ref{prop:p1p2}.)  Therefore, $\textrm{Im}\left(M_{\mS}\right) = E$.  Applying logarithms to~\eqref{eq:prodas} shows that $(\log(a_s))_{s \in \mS}$ belongs to $E$, and thus to $\textrm{Im}\left(M_{\mS}\right)$, so that there exist constants $\alpha_1,\dots,\alpha_d$ and $\beta$ satisfying condition \ref{(ii)} of Theorem~\ref{thm:centralweighting}.

\item{$\ref{(ii)} \boldsymbol \Rightarrow$ Definition~\ref{def:central}.}
For $s \in \mS$ and a walk $w$, we denote by $n_s(w)$ the number of steps of type $s$ contained in $w$. The weight of a walk $w$ of length $n$ starting at $(i_1,\dots,i_d)$, ending at $(j_1,\dots,j_d)$ and staying in $\mathbb R_{\geq 0}^d$ is 
\[ \prod_{s \in \mS}a_s^{n_s(w)} = \beta^{\sum_{s \in \mS} n_s(w)} \prod_{k=1}^d \alpha_k^{\sum_{s \in \mS} \pi_k(s)\,n_s(w)} = \beta^{n} \prod_{k=1}^d \alpha_k^{j_k-i_k}. \]
This number does not depend on $w$, so the weighting is central.\qedhere
\end{description}
\end{proof}

\noindent \textit{Summarizing example.} 
We illustrate how these results allow one to efficiently characterize central weightings by considering the (non-small) step set $\mS = \{(2,2), (1,1), (-1,0), (0,-1)\}$. First, one finds a pair of paths satisfying Proposition~\ref{prop:p1p2}. In this case, we can take $p$ as the sequence of steps $(2,2),(2,2),(2,2),(-1,0),(0,-1)$, and $p'$ as the path formed of five times the step $(1,1)$. Thus, thanks to Proposition~\ref{prop:productweights}, \textit{all} central weightings satisfy
\begin{equation}\label{eq:geom}
     a_{-1,0}a_{2,2}^3a_{0,-1} = a_{1,1}^5.
\end{equation}
In Figure~\ref{fig:paths}, the left-hand side of this equation is encoded by a lattice path in black, while the right-hand side is encoded by a path in white.  

We can compute $\alpha_1,$ $\alpha_2$ and $\beta$  in terms of $a_{2,2}^2$, $a_{0,-1}$ and $a_{1,1}$ by solving the system
\[
\begin{pmatrix}
2 & 2 & 1 \\ 1 & 1 & 1 \\ -1 & 0 & 1 \\ 0 & -1 & 1
\end{pmatrix}
\,
\begin{pmatrix}
\log(\alpha_1) \\ \log(\alpha_2) \\ \log(\beta)
\end{pmatrix}
=
\begin{pmatrix}
\log(a_{2,2}) \\ \log(a_{1,1}) \\ \log(a_{-1,0}) \\ \log(a_{0,-1})
\end{pmatrix}
\]
We find $\alpha_1 = a_{2,2}^2\,a_{0,-1} / a_{1,1}^3$,  $\alpha_2 = a_{1,1}^2 /a_{2,2}\,a_{0,-1}$ and $\beta = a_{1,1}^2 /a_{2,2}$. 

\begin{figure} \center
\includegraphics[width=.32\textwidth]{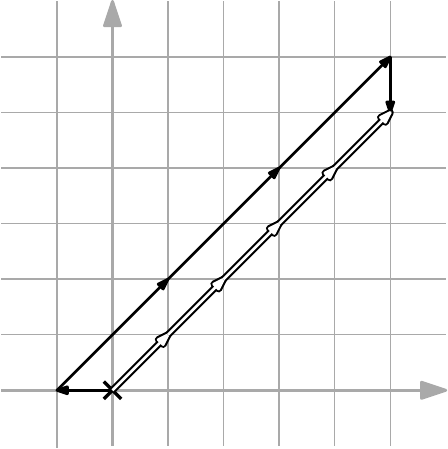} 
\caption{Given the step set $\mS = \{(2,2),(1,1),(-1,0),(0,1)\}$, the path in black and the path in white satisfy the conditions of Proposition~\ref{prop:p1p2}. Remark that they encode the left- and right-hand sides of Equation~\eqref{eq:geom}, respectively.}
\label{fig:paths}
\end{figure}

\subsection{Equivalence classes of weightings}
\label{sec:equiv}
Theorem \ref{thm:centralweighting} can be generalized to models where the original step set is weighted, enabling us to define equivalence classes among weighted models of walks.

Let  $W_{(i_1,\dots,i_d)\rightarrow(j_1,\dots,j_d)}(n)$ be the set of walks of length $n$ starting at $(i_1,\dots,i_d)$, ending at $(j_1,\dots,j_d)$ and staying in $\mathbb R_{\geq 0}^d$.  For $s \in \mS$ and a walk $w$, we denote by $n_s(w)$ the number of steps of type $s$ contained in $w$.
The probability of a given walk $w \in W_{(i_1,\dots,i_d)\rightarrow(j_1,\dots,j_d)}(n)$ among the walks of length $n$ is 
\[ {\prod_{s \in \mS}a_s^{n_s(w)}} \left/ \left({\sum_{w' \in  W_{(i_1,\dots,i_d)\rightarrow(j_1,\dots,j_d)}(n)} \prod_{s \in \mS}a_s^{n_s(w')} }  \right) \right. .\]

\begin{thm}
Let $\mS$ be a set of non-singular integer steps of dimension $d$. Consider two weightings $(a_s)_{s \in \mS}$ and $(a'_s)_{s \in \mS}$ on this step set. Then the following statements are equivalent:
\begin{enumerate}
\item\label{item(i):equivweight}
The probability of any path (starting at the origin, staying in $\mathbb R_{\geq 0}^d$, with fixed length) on the steps in $\mS$ is the same when weighted according to $(a_s)$ or $(a'_s)$;
\item\label{item(ii):equivweight}There exist constants $\alpha_1,\dots,\alpha_d$ and $\beta$ such that the weight assigned to each step $s \in \mS$ is of the form   $a_s = a'_s \, \beta \, \prod_{k=1}^d \alpha_k^{\pi_k(s)}$;
\item\label{item(iii):equivweight}There exist constants $\alpha_1,\dots,\alpha_d$ and $\beta$ such that
\[K_{\mathfrak a}(x_1,\dots,x_d;t) = K_{\mathfrak a'}(\alpha_1 x_1,\dots, \alpha_d x_d; \beta t),\]
where $K_{\mathfrak a}(x_1,\dots,x_d;t)$ denotes the kernel for the weighting $a_s$, and $K_{\mathfrak a'}(x_1,\dots,x_d;t)$  the kernel for the weighting $(a_s')$;
\item\label{item(iv):equivweight}If $(p_s,p'_s)_{s \in \mS \setminus \mathcal T}$ denote pairs of paths satisfying Proposition~\ref{prop:p1p2} then, for every $s \in \mS \setminus \mathcal T$,
\begin{equation}
     \prod_{r \in p_s} \frac{a_r}{a'_r} = \prod_{r' \in p'_s} \frac{a_{r'}}{a'_{r'}}, \label{eq:equivweight}
\end{equation}
where the steps are considered with multiplicity.
\end{enumerate}
Two weightings are said to be \emph{equivalent} if they satisfy one of the above statements.
\label{th:equivweight}
\end{thm}

Thus, a central weighting is a weighting that is equivalent to the unweighted model $(1)_{s \in \mS}$ (note that condition \ref{item(i):equivweight} is equivalent to Definition~\ref{def:central} because every walk of length $n$ has the same probability of occurring in the unweighted case).  The proof of the previous theorem is analogous to the proofs discussed above.  Furthermore, the theorems in the next subsection can be easily generalized to pairs of equivalent (not necessarily central) weightings.

\subsection{Generating function relations}
One of our motivations for central weightings is that they induce very simple relations between the generating functions of a centrally weighted model and an unweighted model. Generalizing from the bivariate case, we define 
\begin{equation}
Q(x_1,\dots,x_d;t) := \sum_{\substack{w \textrm{ walk starting at }(0,\dots,0)\\\textrm{ending at }(i_1,\dots,i_d) \textrm{ with }n\textrm{ steps} \\ \textrm{staying in }\mathbb{N}^d}} x_1^{i_1}\dots x_d^{i_d}t^n
\end{equation}
and
\begin{equation}
Q_{\mathfrak a}(x_1,\dots,x_d;t) := \sum_{\substack{w \textrm{ walk starting at }(0,\dots,0)\\\textrm{ending at }(i_1,\dots,i_d) \textrm{ with }n\textrm{ steps} \\ \textrm{staying in }\mathbb{N}^d}} \, \left( \prod_{s \in \mathcal{S}} a_s^{n_s(w)} \right)x_1^{i_1}\dots x_d^{i_d}t^n,
\end{equation}
where $n_s(w)$ denotes the number of steps $s$ used by the walk $w$.

\begin{proposition} 
Let $Q_{\mathfrak a}(x_1,\dots,x_d;t)$ be the generating function of a model under a central weighting $a_s = \beta \prod_{k=1}^d \alpha_k^{\pi_k(s)}$, and $Q(x_1,\dots,x_d;t)$ the generating function of unweighted walks with the same set of steps. Then 
\begin{equation}
Q_{\mathfrak a}(x_1,\dots,x_d;t) = Q(\alpha_1 x_1,\dots, \alpha_d x_d; \beta t).
\label{eq:Qm=Q}
\end{equation}
\label{prop:Qm=Q}
\end{proposition}
\begin{proof}
Let $(i_1,\dots,i_d,n)$ be a $(d+1)$-tuple of non-negative integers. Extracting the coefficient of $x_1^{i_1}\dots x_d^{i_d}t^n$ in the series $Q_{\mathfrak a}(x_1/\alpha_1,\dots,x_d/\alpha_d;t/\beta)$ gives
\[\sum_{\substack{w\textrm{ walk ending}\\ \textrm{at }(i_1,\dots,i_d)\textrm{ with }n\textrm{ steps}}} \left( \prod_{s \in \mathcal{S}} a_s^{n_s(w)} \right) \alpha_1^{-i_1}\dots\alpha_d^{-i_d}\beta^{-n}.\]
But we know that $n = \sum_{s \in \mS} n_s(w)$ and that $i_k = \sum_{s \in \mS} \pi_k(s)$ for every $k \in \left\{1,\dots,d\right\}$. Hence, by reorganizing the factors, we observe that this coefficient is equal to
\begin{equation*}
     \sum_{\substack{w\textrm{ walk ending}\\ \textrm{at }(i_1,\dots,i_d)\textrm{ with }n\textrm{ steps}}} \prod_{s \in \mathcal{S}} \left(a_{s} \alpha_1^{-\pi_1(s)}\alpha_2^{-\pi_2(s)}\dots\alpha_{d}^{-\pi_{d}(s)}\beta^{-1}\right)^{n_s(w)}.
\end{equation*}
If the weights $a_s$ form a central weighting, then this is equal to the number of walks ending at $(i_1,\dots,i_d)$ with $n$ steps, which is the coefficient of $x_1^{i_1}\dots x_d^{i_d}t^n$ in  $Q_{\mathfrak a}(x_1,\dots,x_d;t)$.
\end{proof}

From this quick observation, and the closure of D-finite functions under algebraic substitution, we obtain the following result.

\begin{cor}\label{thm:bothornone}
The multivariate generating functions for an unweighted model and any weighted version with a rational central weighting are either both D-finite or both non-D-finite. 
\end{cor}

It is interesting to compare this corollary and point \ref{item(iv):equivweight} of Theorem~\ref{th:equivweight} with the works of Kauers and Yatchak~\cite{KaYa15}. Indeed, the models which they show to have finite group (and D-finite generating function) obey relations of the form appearing in Equation~\eqref{eq:equivweight}. This means that the families of walks described in their paper actually correspond to equivalence classes associated to fixed models, with the exception of their Family $0$ which contains an uncountable number of equivalence classes. The representatives of the families of walks appearing in Kauers and Yatchak~\cite{KaYa15} are shown in Figure~\ref{fig:kaya}.

\begin{figure}\center
\includegraphics[width=.9\textwidth]{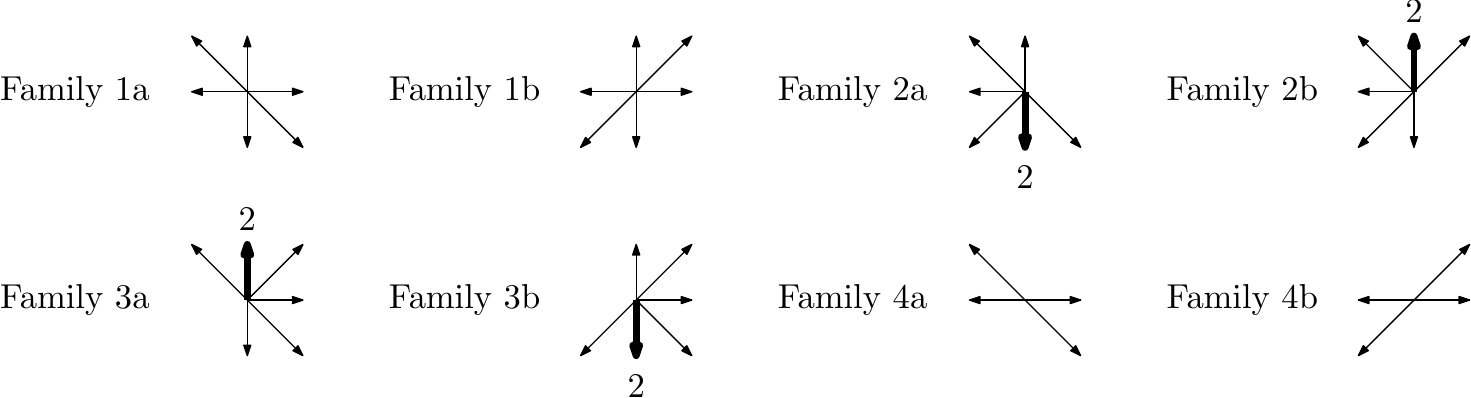} 
\caption{Equivalence class representatives corresponding to the families of walks with D-finite generating functions given by Kauers and Yatchak~\cite{KaYa15}; unlabeled steps have weight $1$. }
\label{fig:kaya}
\end{figure}

\subsection{Excursions}
\label{ss:excursions}

We can also easily deduce from Proposition~\ref{prop:Qm=Q} the following property on excursions, which shows in particular that the subexponential behaviours for the number of excursions for central weightings must all coincide.

\begin{cor} 
\label{thm:excursions}
Given a central weighting of the form $a_s = \beta \prod_{k=1}^d \alpha_k^{\pi_k(s)}$, the number of excursions $e_{\mathfrak a}(n)$ in the weighted model is related to the number of excursions $e(n)$ in the unweighted model by the relation
\[e_{\mathfrak a}(n) = \beta^n\,e(n).\]
\end{cor}
\begin{proof} 
The generating function of excursions is obtained from the generating function of general walks by setting $x_1=\dots=x_{d}=0$. The corollary can then be deduced from Theorem~\ref{thm:centralweighting}.
\end{proof}

\subsection{Conjectures about the converse}
\subsubsection{Statement of the conjectures}

It is natural to ask if the converse of Proposition~\ref{prop:Qm=Q} is also true: are the central weightings the only weightings that satisfy Equation~\eqref{eq:Qm=Q}? This is not clear, and we formulate the following conjecture.

\begin{conjecture} 
\label{conj1}
Let $\mS$ be any step set and ${\mathfrak a}$ be any positive weighting.  If there exist constants $\alpha_1,\dots,\alpha_d,\beta$ such that
\[ Q_{\mathfrak a}(x_1,\dots,x_d;t) = Q(\alpha_1 x_1,\dots, \alpha_d x_d; \beta t), \]
then the weighting ${\mathfrak a}$ is central and $a_s = \beta \, \prod_{k=1}^d \alpha_k^{\pi_k(s)}$.
\end{conjecture}
In the particular case of random walks coming from root systems of Lie algebras, Conjecture~\ref{conj1} was proven by Lecouvey and Tarrago~\cite{LeTa16}.

Conjecture~\ref{conj1} would imply that there is only one function $K$ satisfying the kernel equation for any model. In dimension $d=2$ this would imply that the kernel equation
\[\widetilde K(x,y;t)Q(x,y;t)=xy + \widetilde K(x,0;t)Q(x,0;t) + \widetilde K(0,y;t)Q(0,y;t) - \widetilde K(0,0;t)Q(0,0;t)\]
admits a unique solution $\widetilde K = K$, demonstrating an inherent connection between kernels and generating functions of lattice path models. 

The conjecture can be reformulated purely in terms of the unweighted model.
\begin{conjecture} 
Let $w_{i_1,\dots,i_d}(n)$ be the number of (unweighted) walks with $n$ steps starting at the origin, ending at $(i_1,\dots,i_d)$ and staying in $\mathbb{N}^d$. Assume that there exists a family of constants $(\mu_s)_{s \in \mS}$ such that for every $(d+1)$-tuple $(i_1,\dots,i_d,n)$ the relation
\begin{equation}
\sum_{s \in \mathcal{S}} \mu_s w_{i_1-\pi_1(s),\dots,i_d-\pi_d(s)}(n-1) = 0 \label{eq:conj2}
\end{equation}
holds. Then for every $s \in \mS$, $\mu_s=0$. 
\label{conj2}
\end{conjecture}

\begin{proposition} 
Conjecture \ref{conj2} implies Conjecture \ref{conj1}.
\end{proposition}
\begin{proof}  
The idea is to inject the recurrence relations satisfied by the numbers of walks, such as
\[  w_{i_1,\dots,i_d}(n) =  \sum_{s \in \mathcal S}  w_{i_1-\pi_1(s),\dots,i_d-\pi_d(s)}(n-1), \]
into Equation~\eqref{eq:Qm=Q}. Some basic calculations then show that Equation~\eqref{eq:Qm=Q} implies Equation~\eqref{eq:conj2} with $\mu_s = a_s - \beta  \prod_{k=1}^d \alpha_k^{\pi_k(s)}$. Thus, if Conjecture~\ref{conj2} holds $\mu_s = 0$ for every $s \in \mS$, which proves Conjecture~\ref{conj1}.
\end{proof}

In practice, given a particular model of walks it is easy to verify Conjecture~\ref{conj2}. We can compute the numbers $w_{i_1,\dots,i_d}(n)$ and build the system of equations 
\[\sum_{s \in \mS} \left( w_{i_1-\pi_1(s),\dots,i_d-\pi_d(s)}(n-1) \times \mu_s \right) = 0\] 
with unknowns $\mu_s$, indexed by $(i_1,\dots,i_d,n)$. This system is infinite, but in practice we restrict it to small values of $n$. The null vector is always a solution. If  it is the only one, then Conjecture~\ref{conj2} holds for the model under consideration.

We have computationally verified that this conjecture is satisfied for all the 2D models with small steps in the quadrant. This conjecture can also be proven for certain families of configurations.

\begin{proposition} 
\label{prop:111inS}
Let $\mS$ be a set of small steps such that $(1,1,\dots,1) \in \mS$. Then Conjecture~\ref{conj2} holds for that particular model.
\end{proposition} 

\begin{proof} 
We prove by induction on $j \geq 0$ that for every step $s \in \mS$,
\[\sum_{k=1}^d \pi_k(s) \geq n - j  \quad \Rightarrow \quad \mu_s=0.\]

When $j=0$ the only step satisfying this inequality is $s = (1,1,\dots,1)$. By setting $(i_1,\dots,i_d,n) = (1,\dots,1,1)$, Equation~\eqref{eq:conj2} becomes $\mu_{(1,1,\dots,1)} = 0.$

We now assume that the induction hypothesis is true for $j-1$ and let $s_0 \in \mS$ be a step such that $\sum_{k=1}^d \pi_k(s_0) = n - j$. We will consider Equation~\eqref{eq:conj2} with 
\[ (i_1,\dots,i_d,n) = (1+\pi_1(s_0),1+\pi_2(s_0),\dots,1+\pi_d(s_0),2).\] 
In order to efficiently index the sum in Equation~\eqref{eq:conj2}, we let $\mathcal T$ denote the set of steps $s \in \mS$ such that $w_{1+\pi_1(s_0)-\pi_1(s),\dots,1+\pi_d(s_0)-\pi_d(s),2} \neq 0$.  

For every $s \in \mathcal T$, there exists $s' \in \mS$ such that $\pi_k(s') = 1 + \pi_k(s_0)-\pi_k(s)$ for all $k \in \{1,\dots,d\}$. In particular, by summing the previous identity over all $k \in \{1,\dots,d\}$, we get
\[\sum_{k=1}^d \pi_k(s) = n + (n-j) - \sum_{k=1}^d \pi_k(s').\]
If $s \neq s_0$, then $s' \neq (1,1,\dots,1)$. In that case $\sum_{k=1}^d \pi_k(s') < n$ and $\sum_{k=1}^d \pi_k(s) > n - j$, so we can use the induction hypothesis on $s$ to prove that $\mu_s=0$. Therefore every term that is \textit{not} indexed by $s=s_0$ in Equation~\eqref{eq:conj2} vanishes for the evaluation  $(1+\pi_1(s_0),1+\pi_2(s_0),\dots,1+\pi_d(s_0),2)$. As a consequence we obtain $\mu_{s_0}=0$, which proves the induction hypothesis.
\end{proof}

\subsubsection{The minimal number of steps to prove the conjecture}

If Conjecture~\ref{conj2} turns out to be true, it is also natural to wonder about the minimal bound $N_{\mS}$ on $n$ such that Equation~\eqref{eq:conj2} holding for $n$ up to $N_{\mS}$ implies that $\mu_s=0$ for all $s \in \mS$. In terms of generating functions, $N_{\mS}$ is the minimal integer $n$ such that when the equality in Conjecture~\ref{conj1} holds up to order $n$ then the conclusion in the conjecture holds. 

For example, for the Gouyou-Beauchamps model we can prove that $N_{\mS}=3$. Indeed, with $n=1$ we can only prove that $\mu_{(1,0)} = 0$; then with $n=2$ we only have the restriction $\mu_{(-1,1)}=\mu_{(-1,0)} = 0$. The equality $\mu_{(1,-1)} = 0$ comes at $n=3$. This is due to the fact that a path in $\mathbb R^2_{\geq 0}$ that contains $(1,-1)$ must have a length at least equal to $3$. If $\mS$ has only small steps and $(1,1,\dots,1) \in \mS$, we refer to the proof of Proposition~\ref{prop:111inS} to see that $N_{\mS}=2$.

We have verified that $N_{\mS} \leq 4$ for every 2D model with small steps. Each time that $N_{\mS} = 4$, like for $\mS = \{(1,0),(-1,1),(-1,-1)\}$, this particular value of $N_{\mS}$ is forced by the fact that a particular step must be included in paths of length at least $4$; for $\{(1,0),(-1,1),(-1,-1)\}$ this step is $(-1,-1)$.

It can be shown that $N_{\mS} = 2^d$ is reached for some models with small steps in dimension $d$, such as  
\[\mS = \{(1,0,\dots,0),(-1,1,0,\dots,0),\dots,(-1,-1,-1,\dots,1),(-1,-1,-1,\dots,-1)\}.\]
It thus seems reasonable that $N_{\mS} \leq 2^d$ for every model $\mS$ with small steps in dimension $d$.  Unfortunately, the number~$N_{\mS}$ does not appear to be related to the maximal value over all elements $s\in \mathcal{S}$ of
\[ \min \{ \ell \geq 0 \ | \textrm{ there exists a path of length }\ell\textrm{ containing }s\}.\]
Indeed, there exist some models with a high number of symmetries where~$N_{\mS}$ exceeds that value. It is also for these particular models that Conjecture~\ref{conj2} seems hard to prove.  For example, the $16$-step model~$\mS_{\#}$ in dimension $4$ composed of the $4$ base steps $(1,0,0,0), (0,1,0,0), (0,0,1,0), (0,0,0,1)$ and the $12$ coordinates permutations of the vector $(-1,1,1,0)$ has $N_{\mS_{\#}} = 4$, although every step in $\mS_{\#}$ can be used in a path of length $2$.

\section{When does drift define a universality class?}
\label{sec:Diagrams}

\subsection{Introduction}

It is natural to ask how the drift of a step set affects the asymptotic behaviour of a model, and weighted models provide a good environment to explore this question since the weight parameters can vary continuously and sweep across drift profiles. One consequence of Theorem~\ref{thm:centralweighting} is that in two dimensions every generic central weighting can be defined by only two parameters up to a uniform scaling of weights. In this section, as in Part~\ref{part:GB}, the scaling is chosen in such a way that the constant $\beta$ defined in Theorem~\ref{thm:centralweighting} is equal to $1$.  Here we study how, for a given unweighted two dimensional model, the asymptotic formulas of the model endowed with central weights depends on these two parameters. This allows us to visualize universality classes defined by a model.

\begin{figure}\center
\includegraphics[width=.4\textwidth]{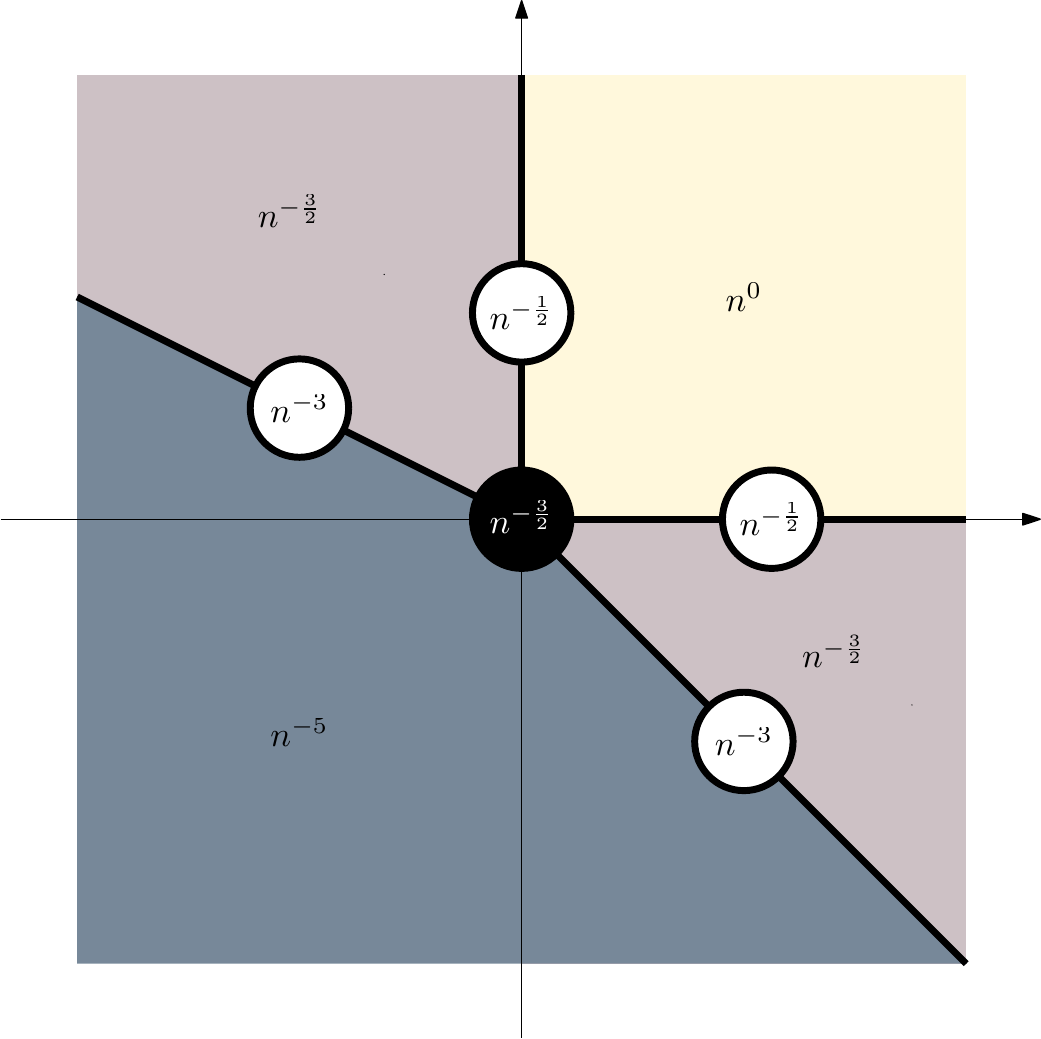}
\caption{The universality classes as a function of drift for the centrally weighted family of the GB model given by \eqref{eq:GB_weighting}: $(1,0)\rightarrow a$, $(-1,0)\rightarrow a^{-1}$, $(-1,1)\rightarrow a^{-1}b$,  $(1,-1)\rightarrow ab^{-1}$. Because the regions are cones the diagram does not change when each weight is multiplied by a constant, meaning this diagram holds for every central weighting.}
\label{fig:drift}
\end{figure}

As discussed above, there are known and conjectured formulas for the exponential growth and critical exponent of different lattice path models. They depend to some extent on the drift of the step set, although in general the critical exponent is not completely determined by the drift.

In the case of equivalence classes of models defined by (scaled) central weights, however, the drift alone is enough to decide the form of the first term of the asymptotic expansion.  That is, in these cases, the drift (as a function of the weights) defines the regions of universality classes. Considering for example the GB model, Figure~\ref{fig:drift} shows this correlation between drift and universality class.

The exact formulas for exponential growth given by Garbit and Raschel~\cite{GaRa16} depend on the minimum of a Laplace transform, defined below, in the cone dual to the cone in which the walks lie. Building on this work, Garbit, Mustapha and Raschel~\cite{GaMuRa16} have a conjecture for the critical exponent. We consider their conjectures in the case of the GB model, and in the next subsection show that these are the same conditions which distinguish cases in the ACSV analysis. This family of examples provides strong support for their conjecture.

The ACSV approach given in Part 1 is applicable to many of the lattice path models with D-finite generating functions. The connection between these two approaches is interesting because the results of Garbit, Mustapha and Raschel are conjectured to hold for any (non-singular) model. This suggests that the structure of the generating functions of the non-D-finite models might resemble the D-finite models: for instance they may be positive parts of (non-rational) meromorphic functions which can be described ``simply'' in terms of each model's steps.

\subsection{The conjecture of Garbit, Mustapha and Raschel} 
\label{ss:gmr}
We now present the conjecture of Garbit, Mustapha and Raschel (specialized for our setting). We begin by introducing some notation. Let $a_{s_1,s_2}$ be the weight of the step $(s_1,s_2)$, and $S(x,y)$ be the inventory 
\[S(x,y)=\sum_{(s_1,s_2)\in \mS} a_{s_1,s_2}\,x^{s_1}y^{s_2}.\]

\begin{rem} 
Although the conjecture of Garbit, Mustapha and Raschel is formulated in terms of the \emph{Laplace transform\/} $L(x,y)=\sum_{(s_1,s_2)\in \mS} a_{s_1,s_2} e^{s_1x+s_2y}$, in order to be consistent with our previous notation we state it in terms of the inventory $S(x,y)$.  Observe that both functions are linked by the relation $S(x,y)=L(\ln(x),\ln(y))$, so the translation is quite straightforward. 
\end{rem}

For a non-singular model, $S(x,y)$ has a unique positive critical point $(\xs, \ys)$ that satisfies $S_x(\xs, \ys)=S_y(\xs, \ys)=0$; note that this critical point is a function of the weights. For the GB model with weights given by Figure~\ref{fig:GB_weights} we have 
\begin{equation}
\label{eq:formula_min_plane}
     (\xs,\ys)=\left(a^{-1},b^{-1}\right).
\end{equation}     
We define the \emph{covariance factor\/}
\begin{equation}
c= \frac{S_{xy}(\xs, \ys)}{\sqrt{S_{xx}(\xs, \ys) S_{yy}(\xs,\ys)}}.
\label{eq:c}
\end{equation}
The value of~$c$ is used to determine the exponential growth, and also appears in the formula for the critical exponent. When we consider only central weightings, the value of~$c$ in Equation~\eqref{eq:c} does not depend on the weights $a_{i,j}$. For example, every centrally weighted GB model has $c=-\frac{\sqrt{2}}{2}$.

\begin{lem} 
\label{lem:c_independent_weights}
For any centrally weighted model, the covariance factor $c$ in \eqref{eq:c} does not depend on the weights.
\end{lem}


Lemma~\ref{lem:c_independent_weights} follows from straightforward calculations after using Theorem~\ref{thm:centralweighting} to write the weight assigned to $(s_1,s_2) \in \mS$ in the form $\beta a^{s_1} b^{s_2}$. We remark that the exponential growth in the reluctant case is given by $S(\xs, \ys)$ and, evaluating this, it is clear how cancellation gives a value independent of the weights.

Let $\mathcal Q$ be the quarter plane defined by 
\[ \mathcal Q = \left\{ (x,y) \in \mathbb R^2 \ \big\vert \ x \geq 1\textrm{ and }y \geq 1 \right\}.\]
Their conjecture states that the universality classes are defined by~$(x^*,y^*)$, the minimum of the inventory $S$ on the cone\footnote{The set $\mathcal{Q}$ is the image of the quarter plane $\left(\mathbb{R}_{\geq 0}\right)^2$ under the transformation $(x,y) \mapsto \left(e^x,e^y\right)$, coming from the fact that the results of Garbit and Raschel~\cite{GaRa16} and Garbit et al.~\cite{GaMuRa16} minimize the Laplace transform $L(x,y)$ instead of the characteristic polynomial $S(x,y)$.  For a walk confined to a more general cone $K$, one must compute the minimum of the characteristic polynomial on the image of the dual cone $K^*$ under the map $(x,y) \mapsto \left(e^x,e^y\right)$. Note that the quarter plane is self-dual. } $\mathcal Q$: \[S(x^*,y^*)=\min_{(x,y) \in \mathcal Q}S(x,y).\] The minimizing point is unique by the strict convexity of the Laplace transform. The minimum is achieved at the critical point $(\xs, \ys)$ when this point is in the region, otherwise it is achieved on the boundary of $\mathcal Q$.

\begin{conjecture}[Garbit, Mustapha and Raschel~\cite{GaMuRa16}]
\label{conj:GMR}
Suppose that $\mS$ is a non-singular step set. Then $[t^n]Q(1,1;t)$ has critical exponent defined by Table~\ref{tab:KR}. 
\end{conjecture}

\begin{table}[h]\center \renewcommand{\arraystretch}{1.5}
\begin{tabular}{c|ccc }
\begin{minipage}{3.1cm}
\ 
\end{minipage}
& 
\begin{minipage}{4cm}
\begin{center}
$\nabla S(x^*,y^*) =0$ \\
(i.e., $(x^*,y^*)=(\xs,\ys)$)
\end{center}
\end{minipage}
& 
\begin{minipage}{4.9cm}
\begin{center}
$S_x(x^*,y^*)=0$ or $S_y(x^*,y^*)=0$

\end{center}
\end{minipage}
& 
\begin{minipage}{3.3cm}
\begin{center}
$S_x(x^*,y^*)>0$ and $S_y(x^*,y^*)>0$
\end{center}
\end{minipage}
\\[2mm]\hline
$(x^*,y^*)= (1,1)$ &

\begin{minipage}{3.5cm}
\vspace*{0.1cm}
\begin{center}
$S(1, 1)^n\,n^{-p_1/2}$ \\ {\color{gray} balanced}
\end{center}
\end{minipage}
&

\begin{minipage}{4.9cm}
\vspace*{0.1cm}
\begin{center}
$S(1,1)^n\,n^{-1/2}$ \\{\color{gray}axial}
\end{center}
\end{minipage}
 & 
\begin{minipage}{3.3cm}
\vspace*{0.1cm}
\begin{center}
$S(1,1)^n n^0$\\{\color{gray}free}
\end{center}
\end{minipage}
 \\[3mm]\hline
$x^*=1$ or $y^*=1$ &  
\begin{minipage}{3.5cm}
\begin{center}
\vspace*{0.1cm}
$S(\xs, \ys)^n\, n^{-p_1/2-1}$\\{\color{gray}transitional}\\ \vspace*{-0.4cm} \ 
\end{center}
\end{minipage}
&
\begin{minipage}{4.9cm}
\begin{center}
$\min\{S(x_1, 1),S(1, y_1)\}^n\,n^{-3/2}$ \\ {\color{gray} directed} \\
\vspace*{-0.4cm} \ 
\end{center}
\end{minipage}
& (not possible)\\[3mm]\hline
$x^*>1$ and $y^*>1$ & 
\begin{minipage}{3.5cm}
\begin{center}
\vspace*{0.1cm}
$S(\xs, \ys)^n n^{-p_1 -1}$\\{\color{gray}reluctant}\\ \vspace*{-0.4cm} \ 
\end{center}
\end{minipage}
& (not possible) & (not possible)\\[3mm]\hline
\end{tabular}

\caption{Conjectured value of the exponential growth and connective constant for two dimensional models in the cone $\mathbb{R}_{\geq0}^2$. Here $(\xs, \ys)$ is the unique positive critical point of the inventory~$S(x,y)$, $S(x,1)$ is minimized at $x_1$, $S(1,y)$ is minimized at $y_1$ and $p_1= \pi/\arccos(-c)$ where $c$ is the covariance factor. Observe that when $(x^*,y^*) = (1,1)$, the vector $(S_x(x^*, y^*),S_y(x^*, y^*))$ is the drift of the model.} 
\label{tab:KR}
\end{table}

\begin{rem}
The exponential growths stated here have been proven by Garbit and Raschel~\cite{GaRa16}. Moreover, Duraj~\cite{Du14} established exact asymptotic estimates in the case of an interior point $(x^*, y^*)$, which corresponds to a \emph{reluctant\/} universality class. 
\end{rem}

Theorem~\ref{thm:main_asymptotic_result} verifies the conjecture for GB walks with a central weighting. For these models, $p_1=\pi/\arccos(\sqrt{2}/2)=4$ and the different regions of universality classes depend either on the drift, or on the signs of $\xx - 1= a^{-1} - 1$ and $\yy - 1 = b^{-1} - 1$. For instance, $\nabla S(x^*, y^*)=(0,0)$ is true if and only if $a \leq 1$ and $b \leq 1$.

Moreover, Table~\ref{tab:KR} gives a formal definition of the six different classes of walks we mentioned in Part~\ref{part:GB} for the GB walks: \textit{balanced}, \textit{reluctant}, \textit{transitional}, \textit{directed}, \textit{axial} and \textit{free} walks.  These adjective describe the qualitative nature of the walks, for example a free walk behaves as if it is unrestrained whereas a reluctant walk behaves like an excursion ending at the origin. In practice, for a central weighting of the form $a_{i,j} = a^i b^j$, it is quite easy to describe these six regions. The free region always corresponds to the walks where the drift is doubly positive; the axial region corresponds to the case where the drift is positive on one coordinate, null on the other one; and the balanced case corresponds to the walk with a zero drift. For the reluctant case, one has to compute the minimum $(x_{s_0},y_{s_0})$ of the \textit{unweighted} inventory $S_0(x,y) = \sum_{(s_1,s_2) \in \mS} x^{s_1} y^{s_2}$. The reluctant region is defined by $a < x_{s_0}$ and $b < y_{s_0}$, while the transitional region is defined by $a < x_{s_0}$ and $b = y_{s_0}$, or $a = x_{s_0}$ and $b < y_{s_0}$. Finally, the remaining region corresponds to the directed case.

\subsection{Connecting back to ACSV}
\label{sec:ACSVConnect}
As we saw in Section~\ref{sec:ACSV} the generating function of the GB model has a strong structure coming from its representation as a rational diagonal, which allows for a detailed analysis. There exists a close relationship between the ACSV and probabilistic approaches for all models with small steps where such a diagonal representation is possible.

\begin{conjecture}
  \label{conj:ACSV}
  If $\mS$ is one of the 19 small step lattice path models with a
  finite non-zero orbit sum then the divisions for the universality
  classes of centrally weighted walk models are given by
  Conjecture~\ref{conj:GMR}.
\end{conjecture}

Although the tools of ACSV should allow one to prove Conjecture~\ref{conj:ACSV}, there are technical difficulties which can arise.  Additionally, one must find the smallest non-zero constant $C_k$ when Proposition~\ref{prop:HighAsm} is used to compute asymptotics around minimal critical points under different weight regimes, which can be a tedious computation for each of the 19 weighted models. Instead, we highlight the connection between Conjecture~\ref{conj:GMR} and ACSV, and mention the ingredients necessary for a proof of Conjecture~\ref{conj:ACSV}.

Recall Equation~\eqref{eq:rho=sup} and Lemma~\ref{lem:expmin}, which imply that the radius of convergence of the GB generating function $Q_{\mathfrak a}(t)$  is the supremum of $\vert xyt\vert $ over all minimal critical points. It is easy to verify this for most\footnote{The rational functions corresponding to some models have denominators which additionally contain one of the polynomials $x^2+a^2$ and $x^2+ax+a^2$ as factors.  These polynomials do not introduce new critical points but may affect minimality of existing critical points for certain values of $a$, complicating the analysis.} of the 19 models with small steps, since their generating functions can be usually be expressed as the diagonal part of a rational function of the form  
\[\frac{G(x,y,t)}{(1 - xytS(\overline x,\overline y))(1-x)(1-y)},\] 
where $S$ is the inventory of the walk model. By Proposition~\ref{prop:findcritpts}, a critical point must belong to the variety $1 - xytS(\overline x,\overline y)=0$.  Hence, by eliminating $t$ we reduce the problem to the study of the supremum of $\vert S(\overline x,\overline y)^{-1}\vert $ over all minimal points; i.e., points satisfying $0<\vert x\vert  \leq 1$ and $0<\vert y\vert \leq 1$.

For non-singular models this supremum is actually achieved at a point $(x,y) = ((x^*)^{-1},(y^*)^{-1})$ in $\mathcal{Q}$. We then observe a correspondence between the different regions of universality (see Table~\ref{tab:KR}) and the different strata (see Subsection~\ref{ss:critpts}):
 \begin{itemize}
 	\item If the minimum of $S$ is achieved in the interior of $\mathcal Q$, the walk model is reluctant. Thus, since $\left\vert S(\overline x,\overline y)^{-1}\right\vert $ is maximized over all minimal critical points at a point $(x,y)$ satisfying $\vert x\vert < 1$ and $\vert y\vert  < 1$, the contributing points must belong to $\mV_1$. 

	\item If the minimum of $S$ is achieved at a point such that $x=1$ or $y=1$ (but not both), then the walk model is either directed, or transitional. Similarly, the two corresponding strata are $\mV_{12}$ or $\mV_{13}$, depending whether the maximum is achieved on the $x$-axis, or the $y$-axis.
 
	\item Finally, if the minimum is achieved at $(1,1)$, then we are either in the free, axial, or balanced case. The contributing point belongs thus to $\mV_{123}$.
 \end{itemize}
 By this reasoning we recover the exponential growths in Table~\ref{tab:KR}. To find the critical exponents we must follow the approach described in Subsection~\ref{ss:asympcontr} and apply Proposition~\ref{prop:HighAsm}. This is the step that permits the distinction between the directed and transitional cases, and between the free, axial and balanced cases. As in the GB case, there may be simplifications between numerator and denominator in the borderline cases (i.e., transitional, axial and balanced).

 In summary, ACSV allows one to verify Conjecture~\ref{conj:GMR} with explicit computations, including leading asymptotic constants. There is also a strong connection between the three rows of Table~\ref{tab:KR} and the different strata. However, it is hard to predict from ACSV without doing the actual calculations the values of the critical exponents, and the distinction induced by the three columns of Table~\ref{tab:KR}.  It is our hope to understand Conjecture~\ref{conj:GMR} better via a more advanced ACSV approach, and extend the proof to non-D-finite generating functions.

\subsection{Drift diagrams for universality classes}
\label{sec:SimpleandWeight}
 We end with some additional examples of diagrams for universality classes, to compare against the Gouyou-Beauchamps model.
\begin{example}
Consider the model defined by the step set $\{(1,0), (-1, 1), (0,-1) \}$, dubbed the \emph{Tandem model} due to its applications in queuing theory. As it has only three steps, all positive weightings are central. We consider the following family of central weightings: 
\begin{equation}
    (1,0)\mapsto a,\qquad (-1,1)\mapsto a^{-1}b,\qquad (0,-1)\mapsto  b^{-1}.
    \label{eq:tandem}
\end{equation}
As in the GB model, there is a breakdown of cases for the contributing singularities based on where one finds critical minimal points. The drift of the weighted model is $(a-b/a, b/a-1/b)$, and the critical exponent of the model can be determined directly from it.  Of note is that the region defining the reluctant universality class is not a cone.

\begin{figure}\center
\includegraphics[width=.4\textwidth]{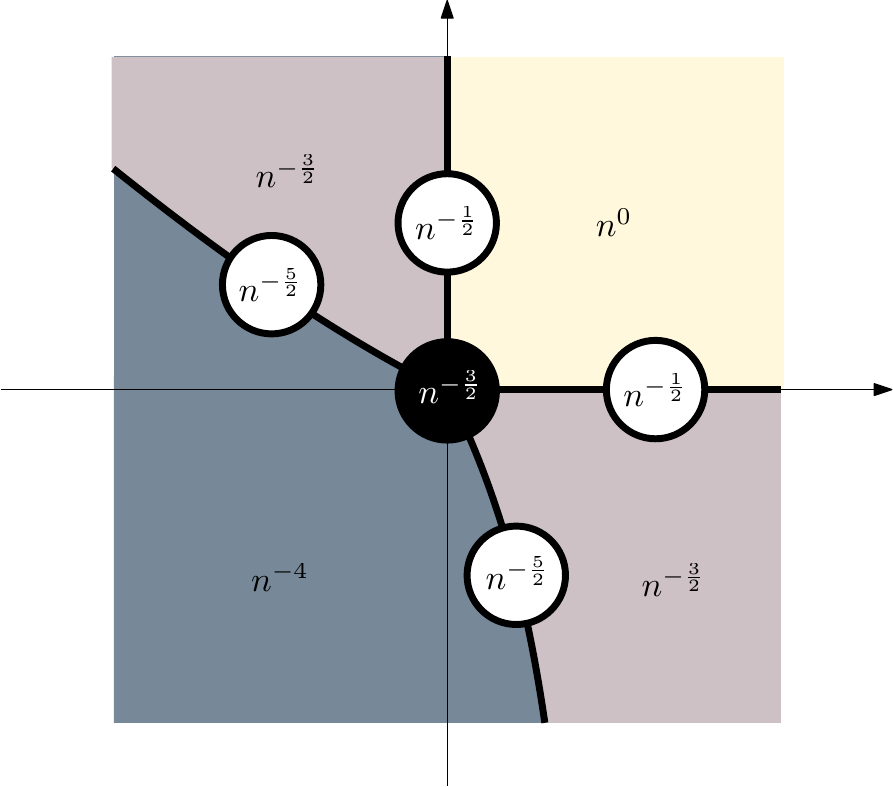}
\caption{The universality classes as a function of drift for the centrally weighted family of the Tandem model given by \eqref{eq:tandem}: $(1,0)\rightarrow a$, $(-1,1)\rightarrow a^{-1}b$, $(0,-1)\rightarrow b^{-1}$.}
\end{figure}
\end{example}

Although the drift diagram for the Tandem model does not separate into regions which are cones, there are models where this occurs. Consider for example the Gessel model with weights 
\[ (-1,0)\mapsto a^{-1}, \qquad (1,0)\mapsto a,\qquad (1,1)\mapsto ab,\qquad (-1,-1)\mapsto  a^{-1}b^{-1}.\]
It can be shown that this model is reluctant if and only if $a < 1$ and $b < 1$. In terms of the drift $(d_x,d_y)$, this means that the reluctant region is defined by the cone $2d_y - d_x < 0$ and $2d_x - d_y<0$.

\begin{figure}\center
\mbox{}\hfill \includegraphics[width=0.4 \textwidth]{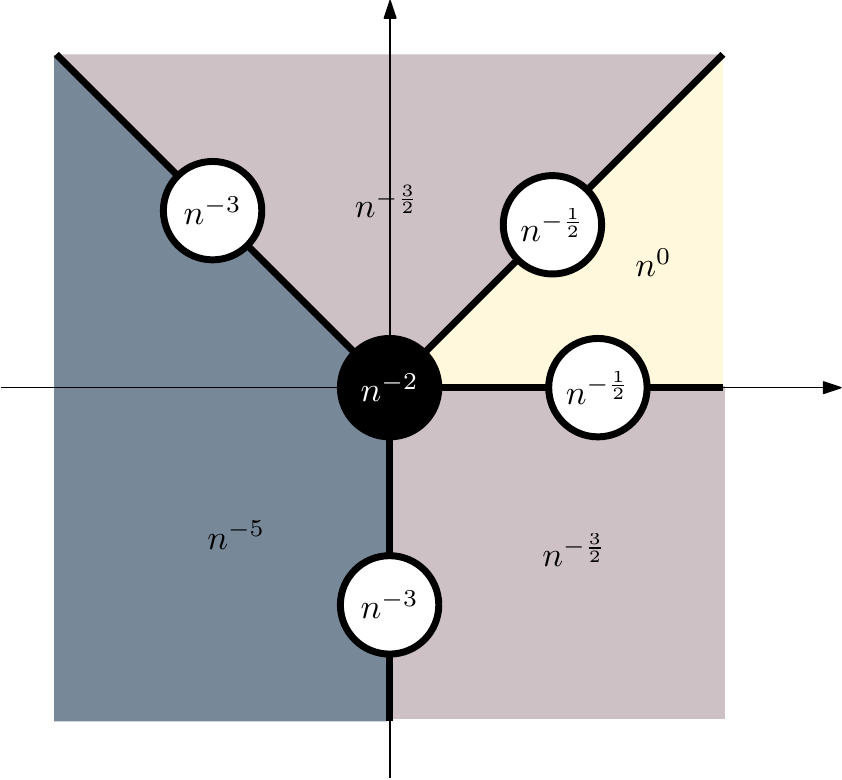} \hfill
\includegraphics[width=0.4 \textwidth]{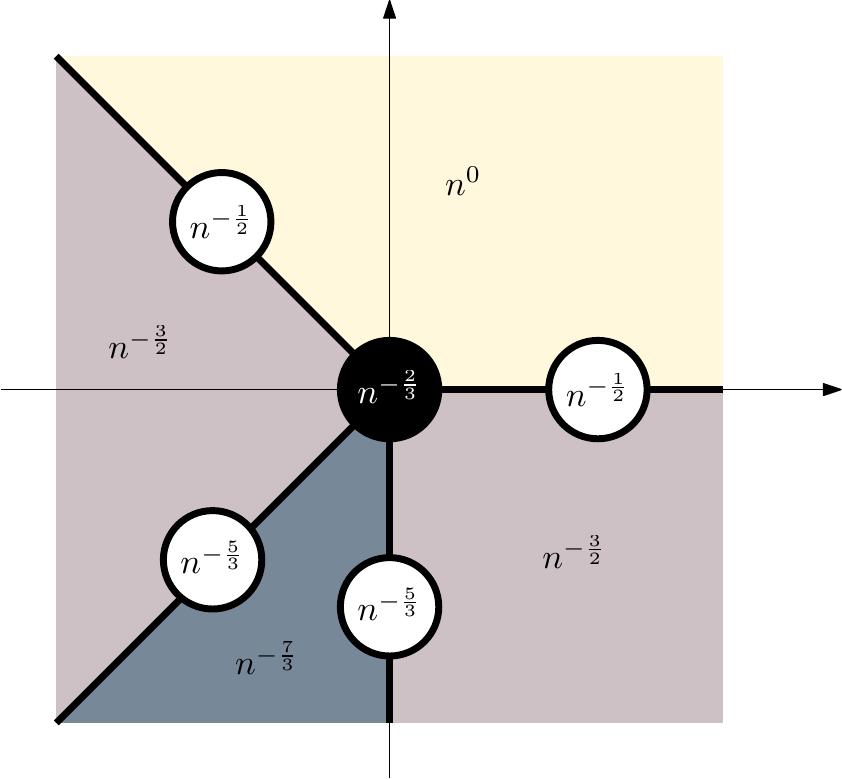} \hfill\mbox{}
\caption{Universality classes defined by drift for simple walks in the cone $\frac{\pi}{4}$ (GB walks) and in the cone $\frac{3\pi}{4}$ (Gessel walks).}
\label{fig:simplewalks}
\end{figure}

Gouyou-Beauchamps and Gessel walks in the quarter plane are in bijection with simple walks (i.e., walks with step set $\{(1, 0), (-1,0), (0, 1), (0,-1)\}$) staying in cones of angles $\pi/4$ and $3\pi/4$, respectively. When examining the drift diagram of these models in these new cones (see Figure~\ref{fig:simplewalks}) one observes that the reluctant region has an angle equal to $3\pi/4$ and $\pi/4$, respectively.   This means that the regions corresponding to directed walks are cones of angle $\pi/2$, a phenomenon which also occurs for the simple walk itself. The reluctant region is thus the polar cone to the free region. We do not know any simple explanation of this fact, however we can prove that these three models are the only models of simple walks (on a square lattice) in a two dimensional cone that have this structure.

Nice decompositions into cones also occur for the drift diagram of Double Tandem and Double Kreweras models of Bousquet-M{\'e}lou and Mishna~\cite{BoMi10}, when we describe them as simple walks on a triangular lattice in cones of angles $\pi/3$, and $2\pi/3$. As an open question, we wonder if there exists a general explanation to this phenomenon of simple walks in different cones.  For example, can it happen in higher dimensions?

\section{Conclusion}
\label{sec:conclusion}
Central weightings provide a simple but efficient way to understand the shifts in asymptotic behaviour of weighted lattice path models. They put the conjecture of Garbit, Mustapha and Raschel in a new light, and have strong connections to areas of combinatorics, probability theory, and the random generation of discrete structures.

All of the small step weighted D-finite models in two dimensions found by Kauers and Yatchak~\cite{KaYa15} consist of a centrally weighted equivalence class, except for their family 0 (which consists of an infinite number of equivalence classes). It is thus natural to wonder how many equivalence classes of D-finite models---with larger steps or in higher dimensions--- exist, and how many equivalence classes contain an unweighted model (families 2 and 3 of Kauers and Yatchak do not contain an unweighted model).

Furthermore, there are still many open problems for walks with small steps in a quadrant.  For example, it is still unknown for many unweighted lattice path models whether the univariate generating function $Q(1,1;t)$ of walks with no constraint on the endpoint is D-finite. For the models in which this question is still open, one would expect the univariate generating function to be non-D-finite, since the trivariate generating function $Q(x,y;t)$ is not~\cite{KuRa12,BoRaSa14}.  With our formalism we can see that there exist models with an infinite number of non-D-finite evaluations $Q(\alpha_1,\alpha_2;t)$ of the generating function.  Indeed, any central weighting of a model such that $\arccos(-c)$ is not commensurate with $\pi$ (see \eqref{eq:c} for the definition of $c$) will work by a result of Duraj~\cite{Du14}, which implies that the critical exponent for the model is not rational\footnote{Any D-finite lattice path generating function is an example of a \emph{G-function}, meaning in particular that the number of lattice walks cannot have asymptotic growth of the form $C \rho^n n^{\alpha}$, with $\alpha$ irrational.  More details can be found in the work of Bostan, Raschel and Salvy~\cite{BoRaSa14} which used this argument to show the non-D-finiteness of excursion generating functions $Q(0,0;t)$.}. Proposition~\ref{prop:Qm=Q} relates the non-D-finite generating function of the weighted model and the evaluated unweighted generating function $Q(\alpha_1,\alpha_2;t)$.

\begin{cor}
\label{thm:infdfinite}
There are an infinite number of two dimensional weighted lattice models with non-D-finite generating function. 
\end{cor}

Finally, the present work could be useful for the studies of specific models of walks. As an example, it would be interesting to understand the model with simple steps ($(\pm 1,0),(0,\pm 1)$) in cones with an angle $\theta$, potentially not commensurable with $\pi$. When is such a model D-finite and what are the different universality classes?

\paragraph{Acknowledgments} 
KR would like to thank C.\ Lecouvey and P.\ Tarrago for interesting discussions concerning their work~\cite{LeTa16}. KR would also like to thank R.\ Garbit for enlightening discussions. MM and JC were partially supported by Natural Sciences and Engineering Research Council of Canada (NSERC) Discovery Grant 31-611453. JC was also supported by the Pacific Institute for the Mathematical Sciences and by the French "Agence Nationale de la Recherche", project A3 ANR-08-BLAN-0190. SM was partially supported by an NSERC Canadian Graduate Scholarship and a David R. Cheriton Graduate Scholarship. Finally, the authors would like to thank the anonymous reviewers for their valuable comments. 

\bibliographystyle{plain}
\bibliography{bibl}

\appendix

\section{Values of the harmonic function}
\label{sec:values_harmonic_function}
We now give the values of the universal constant $\kappa$ and harmonic function $V^{[n]}(i,j)$ for the various models.

\paragraph{Balanced $a=b=1$} Here $\kappa = \frac{8}{\pi}$ and
\[ V^{[n]}(i,j) = \frac{(i+1)(j+1)(i+j+2)(i+2j+3)}{6}. \]

\paragraph{Free $\left(\sqrt{b}<a<b\right)$} 
Here $\kappa = 1$ and 
\begin{align*} 
V^{[n]}(i,j) = a^{-(4+2i+2j)}b^{-(2+2j)}&\left(\left(a^{1+j}-1\right)\left(a^{1+j}+1\right)\left(a^{2+i+j}-b^{2+i+j}\right)\left(a^{2+i+j}+b^{2+i+j}\right)b^{-i-1} \right.  \\
&\left. \qquad\qquad - \left(a^{2+i+j}-1\right)\left(a^{2+i+j}+1\right)\left(a^{1+j}-b^{1+j}\right)\left(a^{1+j}+b^{1+j}\right) \right).
\end{align*}

\paragraph{Reluctant $\left(a<1,b<1\right)$}
Here $\kappa = \frac{64}{\pi(b-1)^4}$ and  
\begin{align*} 
V^{[n]}(i,j) = \frac{(1+j)(1+i)(3+i+2j)(2+i+j)}{a^ib^j} &\left(\frac{a^2b^2+a^2b-4ab+b+1}{(a-1)^4}\right. \\
&\qquad+ \left. (-1)^{n+i}\frac{a^2b^2+a^2b+4ab+b+1}{(a+1)^4}\right).
\end{align*}

\paragraph{Directed 1 $\left(b>1, \sqrt{b}>a\right)$}
Here $\kappa = \frac{\sqrt{2}}{\sqrt{\pi}b^2}$ and  
\[ V^{[n]}(i,j) = \left(\frac{b^{3+i+2j}(1+i)+\left(b^{1+j}-b^{2+i+j}\right)(3+i+2j)-i-1}{a^ib^{i/2+2j}}\right) \left(\frac{1}{(\sqrt{b}-a)^2}+\frac{(-1)^{i+n}}{(\sqrt{b}+a)^2}\right). \]

\paragraph{Directed 2 $\left(a>1, a>b\right)$}
Here $\kappa = \frac{(a+1)^3\sqrt{a}}{2\sqrt{\pi}(a-b)^2}$ and  
\[ V^{[n]}(i,j) =  (2+i+j)\left(a^{-2-j}-a^j\right)b^{-j}a^{-1-i} + (1+j)\left(1-a^{-4-2i-2j}\right)b^{-j}a^j.\]

\paragraph{Axial 1 $\left(a=b>1\right)$}
Here $\kappa = \frac{b+1}{\sqrt{b\pi}}$ and  
\[ V^{[n]}(i,j) =  (j+1)\left(1-b^{-2(2+i+j)}\right)+b^{-i-1}(i+2+j)\left(b^{-2(1+j)}-1\right). \]

\paragraph{Axial 2 $\left(b=a^2>1\right)$}
Here $\kappa = \frac{\sqrt{2}}{a^6\sqrt{\pi}}$ and  
\[ V^{[n]}(i,j) =  \left(a^6-a^{-2i-4j}\right)(1+i)+\left(a^{2-2i-2j}-a^{4-2j}\right)(3+i+2j). \]

\paragraph{Transitional 1 $\left(a=1,b>1\right)$}
Here $\kappa = \frac{16}{3\pi(1-b)^2}$ and  
\[ V^{[n]}(i,j) =  (j+1)(i+1)(i+3+2j)(i+2+j)b^{-j}. \]

\paragraph{Transitional 2 $\left(b=1,a>1\right)$}
Here $\kappa = \frac{8}{3\pi}$ and  
\[ V^{[n]}(i,j) =  a^{-i}(j+1)(i+1)(i+3+2j)(i+2+j)\left(\frac{1}{(1-a)^2} + \frac{(-1)^{n+i}}{(1+a)^2} \right). \]

\end{document}